\documentclass[12pt]{amsart}
\usepackage[english]{babel}
\usepackage{amssymb,amsmath,verbatim}
\usepackage{amscd}
\usepackage[all]{xy}
\usepackage[usenames]{color}

\newtheorem{Th}{Theorem}[section]

\newtheorem{Le}[Th]{Lemma}
\newtheorem{Cj}[Th]{Conjecture}
\newtheorem{Cr}[Th]{Corollary}
\newtheorem{Prop}[Th]{Proposition}
\numberwithin{equation}{section}

\theoremstyle{remark}
\newtheorem{Rm}[Th]{Remark}

\begin{document}

\title{Galois invariants of $K_{1}$-groups of Iwasawa algebras.}

 \author{Dmitriy Izychev }
\address{Universit\"{a}t Heidelberg\\ Mathematisches Institut\\
Im Neuenheimer Feld 288\\ \newline 69120 Heidelberg, Germany.} \email{dizychev@mathi.uni-heidelberg.de}

\author{Otmar Venjakob}%
\address{Universit\"{a}t Heidelberg\\ Mathematisches Institut\\
Im Neuenheimer Feld 288\\ \newline 69120 Heidelberg, Germany.} \email{venjakob@mathi.uni-heidelberg.de}
\urladdr{http://www.mathi.uni-heidelberg.de/\textasciitilde otmar/}
\thanks{We acknowledge support  by the DAAD, ERC and
DFG}




\date{\today}%
\maketitle \thispagestyle{empty}


\section*{Introduction}

\quad Let $S$ be an arbitrary ring with unit, then $K_{1}(S)$ is defined by the exact
sequence
\begin{equation*}
1 \rightarrow  E(S) \rightarrow  GL(S)\rightarrow K_{1}(S)\rightarrow 1,
\end{equation*}
where $ GL(S)=\underset{\rightarrow}{\lim}\; GL_{n}(S)$ denotes the general linear group
of $S$ and $ E(S)$ denotes the subgroup of elementary matrices over $S$.

In \cite[Rem.\ 3.4.6]{FK} Fukaya and Kato formulate the following expectation, which plays an
important role in the definition of so called $\varepsilon$-isomorphisms concerning the local
Iwasawa theory of $\varepsilon$-constants: For an {\it adic ring} $\Lambda$ in the sense of 1.4.1 in (loc.\ cit.) the sequence
\[\xymatrix@C=0.5cm{
  0 \ar[r] & K_1(\Lambda) \ar[rr]^{ } && K_1(\tilde{\Lambda})  \ar[rr]^{1-\varphi_*} && K_1(\tilde{\Lambda}) \ar[r] & 0 }\]
 should be exact, where
 $\tilde{\Lambda}:=\widehat{\mathbb{Z}_p^{ur}}\hat{\otimes}_{\mathbb{Z}_p}\Lambda$ denotes the
 completed tensor product
 and $\varphi$ denotes the Frobenius morphism acting on the ring of integers $\widehat{\mathbb{Z}_p^{ur}}$ of the $p$-adic completion $\widehat{\mathbb{Q}^{ur}_{p}}$ of the maximal unramified extension of
$\mathbb{Q}_{p}.$ For any finite group $G$ and $\Lambda:=\mathbb{Z}_{p}[G]$ this amounts to the
    statement
\begin{equation*}
i_{*}:K_{1}(\mathbb{Z}_{p}[G])\cong K_{1}(\widehat{\mathbb{Z}_p^{ur}}[G])^{\varphi=id},
\end{equation*} where $i_{*}$ is induced by the inclusion
$i:\mathbb{Z}_{p}\rightarrow\widehat{\mathbb{Z}_p^{ur}}$, and the Frobenius map $\varphi$     acts
coefficientwise on $\widehat{\mathbb{Z}_p^{ur}}[G]$  and hence on the $K_{1}$-group.

The original motivation  of this note was to show Fukaya and Kato's expectation in this specific
case. A  more general question  would rather be whether the following statement
\begin{equation}  i_{*}:K_{1}(S^{\Delta}[G])\cong K_{1}(S[G])^{\Delta}, \label{eq:1}
\end{equation}
holds whenever  $S$ is a ring and $\Delta$ is a group acting on $S$ by ring automorphisms. 
But surprisingly, it turns out, that neither of the above statements does   hold in general (see
3.2). In this paper we restrict our attention to the case, where $\Delta$ is the Galois group of
 some   algebraic field extension related  to the extension of a specific class of rings $S/S^{\Delta}$ contained in the completion $\mathbb{C}_p$ of $\overline{\mathbb{Q}_p}$. We  obtain partial results
toward (a corrected version of) \eqref{eq:1}, see Theorem \ref{prop:4}. In particular, we show that in general the following sequence is exact
 \[\xymatrix@C=0.5cm{
  1 \ar[r] & SK_1(\mathbb{Z}_p[G]) \ar[rr]^{ } && K_1(\mathbb{Z}_p[G]) \ar[rr]^{i_* } && K_{1}(\widehat{\mathbb{Z}_p^{ur}}[G])^{\varphi=id}\ar[rr]^{ } && 1
  }\] and induces  an isomorphism for the rational $K$-groups \[K_1(\mathbb{Z}_p[G])_\mathbb{Q} \cong K_1(\widehat{\mathbb{Z}_p^{ur}}[G])_\mathbb{Q}^{\varphi=id}.\]

If $S$ is a finite algebraic extension of $\mathbb{Z}_p$ and $SK_1(S[G])=1$, the isomorphism
\eqref{eq:1} reduces to Galois descent of the determinantal image:
\[i_{*}:\mathrm{Det}(S^\Delta[G])\cong \mathrm{Det}(S[G])^\Delta\] as has been proved by M.\ Taylor
in the case where $S$ is unramified. But the case of infinite extensions of $\mathbb{Z}_p$ and
infinite groups $\Delta$ seems not to be covered in the literature, not even by the fairly general
recent treatment \cite{CPT}, where only finite group actions are considered, as was pointed out to
us by M.\ Taylor. Actually one has to check that the techniques of integral group logarithms extend
to this situation, either  by extending Taylor's original definition as pursued in (loc.\ cit.) or
by using Snaith' version in \cite{Sn} - both in the case of $p$-groups- and then use standard
induction techniques to reduce the general case of finite groups to it, as e.g.\ in \cite{Fr.}.
Both approaches work and for the convenience of the reader we show that the methods of \cite{CPT}
extends easily to our setting recalling the main steps of their proof, but noting that for ramified
extensions Snaith's construction might be better adapted.

The reason for the defect in \eqref{eq:1} relies on the surprising vanishing
\[SK_1(\widehat{\mathbb{Z}_p^{ur}}[G])=1\] for {\it all finite groups} $G,$ in particular $SK_1$ - in contrast to the Det-part - does not
have good Galois descent in general, see Corollary \ref{cr:33} for a more precise statement.

The paper is organized as follows: In the first section we recall for the convenience of the reader
Galois descent results for group rings with coefficients in local or global fields using Fr\"{o}hlich's
Hom-description. In the second section, the heart of the paper, we first concentrate on descent
results for the Det-part. In particular we obtain a rather general result not only for finite
groups, but also for compact $p$-adic Lie groups and their Iwasawa algebras, which turns out to be
quite useful in number theory, see   \cite{BV}. Then we deal with the $SK_1$-part recalling and
generalizing results from \cite{O.}. Altogether both parts lead to the desired descent description
for $K_1.$ Finally, in the last section we derive similar descent results over the corresponding
residue class fields.

We are very grateful to Victor Snaith and  Martin Taylor for various discussions and for informing
us about the state of the art regarding the descent problem.

\section{The case of ``local'' and ``number'' fields.}

\quad The goal of this section is to prove the following theorem which is certainly known to
experts but for lack of a reference we treat it here, because it forms the prototype for the
descent results in the integral cases later.

\begin{Th} \label{th:9}
Let $L$ be a finite Galois extension of $\mathbb{Q}_{p}$ and $M$ be either an arbitrary (possibly infinite) Galois extension $M^{0}$ of $\mathbb{Q}_{p}$ or the $p$-adic completion of a Galois extension $M^{0}$ of finite absolute ramification index over $\mathbb{Q}_{p}$, so that $\mathbb{Q}_{p}\subseteq L\subseteq M\subseteq \mathbb{C}_{p}$. Further, let $\Delta=Gal(M^{0}/L)$, such that $M^{\Delta}=L$ (see Remark \ref{rm:27}). Let $\Gamma$ be a finite group. Then
\begin{equation} \label{eq:8}
i_{*}:K_{1}(L[\Gamma])\cong K_{1}(M[\Gamma])^{\Delta},
\end{equation}
where $\Delta$ acts on the $K_{1}$-group coefficientwise.
\end{Th}

For this we need the following

\begin{Prop} \label{prop:3}
Let $N$ be either an arbitrary (possibly infinite) algebraic extension of $\mathbb{Q}_{p}$ or the completion of an algebraic extension of finite absolute ramification index over $\mathbb{Q}_{p}$. Let $\Gamma$ be a finite group. Then \newline
(i) The map $i_{*}:K_{1}(N[\Gamma])\rightarrow K_{1}(\overline{N}[\Gamma])$ is injective, \newline
(ii) If $N$ is a finite Galois extension of $\mathbb{Q}_{p}$ and $G_{N}=Gal(\overline{N}/N)$ is the absolute Galois group, then
\begin{equation*}
i_{*}:K_{1}(N[\Gamma])\cong K_{1}(\overline{N}[\Gamma])^{G_{N}}.
\end{equation*}
\end{Prop}
\begin{proof}
The first statement is well known for the local fields, i.e. finite extensions of $\mathbb{Q}_{p}$, and more generally for the perfect discrete valued fields (see Proposition 2.8 in \cite{Q.}; Theorem 1 and the Remark after Theorem 2 in \cite{NM}). The infinite algebraic extensions can always be written as a direct limit of their finite subextensions. Since the direct limit is exact on the category of abelian groups and $K_{1}$ commutes with the direct limit (see \cite{R.}, Exercise 2.1.9), (i) is true for infinite algebraic extensions. This completes the proof of (i).

Let $R_{\Gamma}=R_{\Gamma}(\overline{N})$ denote the Grothendieck group of finitely generated
$\overline{N}[\Gamma]$-modules. Alternatively $R_{\Gamma}$ will be viewed as the group of virtual
$\overline{N}$-valued characters of $\Gamma$. Since $\overline{N}$ is algebraically closed,
$R_{\Gamma}$ is a free abelian group on the irreducible characters.

Using the Wedderburn's decomposition of $\overline{N}[\Gamma]$ we get an isomorphism of $G_{N}$-modules
\begin{equation} \label{eq:9}
K_{1}(\overline{N}[\Gamma])\cong \underset{\chi}{\prod}\;\overline{N}^{\times}\cong
\mathrm{Hom}(R_{\Gamma},\overline{N}^{\times}),
\end{equation}
where $\chi$ are irreducible $\overline{N}$-valued characters and the action of $G_{N}$ on the
Hom-group is given by the actions on $R_{\Gamma}$ and on $\overline{N}^{\times}$ in the standard way
\begin{equation*}
f^{g}(\chi)=(f(\chi^{g^{-1}}))^{g},\forall f\in \mathrm{Hom}(R_{\Gamma},\overline{N}^{\times}),\;g\in
G_{N},\;\chi\in R_{\Gamma}.
\end{equation*}

From \cite{T.} (part 1, paragraph 2)  we obtain the corresponding Hom-description for $K_{1}(N[\Gamma])$
\begin{equation} \label{eq:10}
K_{1}(N[\Gamma])\cong \mathrm{Hom}_{G_{N}}(R_{\Gamma},\overline{N}^{\times}).
\end{equation}
Then the second statement is obvious, as $i_{*}$ is a Galois homomorphism.
\end{proof}

\begin{Rm} \label{rm:27}
We fix an embedding $\overline{\mathbb{Q}_{p}}\hookrightarrow \mathbb{C}_{p}$. With the previous notation we have by the theorem of Ax-Sen-Tate
\begin{equation*}
M^{\Delta}=(M^{0})^{\Delta}=L\text{ and }\mathcal{O}_{M}^{\Delta}=(\mathcal{O}_{M^{0}})^{\Delta}=\mathcal{O}_{L}.
\end{equation*}
In the following we will always use this theorem for the case of completions.
\end{Rm}

Now let $L$ and $M$ be as in Theorem \ref{th:9}. We introduce a commutative diagram
\begin{equation} \label{eq:11}
\begin{array}{ccccc}
K_{1}(L[\Gamma]) & \rightarrow & K_{1}(\overline{L}[\Gamma]) & \tilde\rightarrow & \underset{\chi}{\prod}\;\overline{L}^{\times}\\
\Big\downarrow{i_{*}} & & \Big\downarrow & & \Big\downarrow\\
K_{1}(M[\Gamma]) & \rightarrow & K_{1}(\overline{M}[\Gamma]) & \tilde\rightarrow & \underset{\psi}{\prod}\;\overline{M}^{\times},\\
\end{array}
\end{equation}
where $\chi$ and $\psi$ are irreducible characters of $\Gamma$ with values in $\overline{L}^{\times}$ and
$\overline{M}^{\times}$ respectively. The rows are injective by the Proposition \ref{prop:3}. The right
hand side column is injective as $\overline{L}^{\times}\subseteq \overline{M}^{\times}$ and each $\psi$ being a
character of a finite group is the composition of one of the characters $\chi$ with the inclusion
map $i:\;\overline{L}\rightarrow\overline{M}$. It follows, that the left hand side column is also
injective. Note also, that \[R_\Gamma(\overline{L})\cong R_\Gamma(\overline{M}),\] which we take as
an identification henceforth.

Taking invariants under the action of $\Delta$, which is a left exact functor, we obtain the first inclusion of Theorem \ref{th:9}
\begin{equation} \label{eq:12}
K_{1}(L[\Gamma])\subseteq(K_{1}(M[\Gamma]))^{\Delta}.
\end{equation}

On the other hand, from the diagram \eqref{eq:11} and the isomorphism \eqref{eq:9} we deduce
\begin{equation*}
(K_{1}(M[\Gamma]))^{\Delta}=(K_{1}(M[\Gamma]))^{G_{L}}\subseteq (K_{1}(\overline{M}[\Gamma]))^{G_{L}}\cong
\end{equation*}
\begin{equation*}
\mathrm{Hom}_{{G_{L}}}(R_{\Gamma}(\overline{M}),\overline{M}^{\times})\cong
\mathrm{Hom}_{{G_{L}}}(R_{\Gamma}(\overline{L}),\overline{L}^{\times}).
\end{equation*}
The last isomorphism is given by Remark \ref{rm:27}.

By the Proposition \ref{prop:3} (ii) the last Hom-group is equal to $K_{1}(L[\Gamma])$, thus we get the second inclusion of Theorem \ref{th:9}, and hence the isomorphism \eqref{eq:8}.

\begin{Rm} \label{rm:7}
Unfortunately we cannot prove Theorem \ref{th:9} in full generality, i.e., for $M$ being the
completion of an arbitrary Galois extension of $\mathbb{Q}_{p}$. For instance, it is not known to
us, whether the theorem holds for $M$ being the completion of the infinite purely ramified
extension $\mathbb{Q}_{p}(\mu_{p^{\infty}})$ of $\mathbb{Q}_{p}$.
\end{Rm}

\begin{Rm}
The proof above also works for ``number'' fields, i.e. algebraic (possible infinite) extensions of
$\mathbb{Q}$. We just have to replace $\mathbb{Q}_{p}$ by $\mathbb{Q}$ in Theorem \ref{th:9}. Then,
letting $L$ be a finite Galois extension of $\mathbb{Q}$ and M be an arbitrary (possible infinite)
Galois extension of $\mathbb{Q}$, we follow the proof of Theorem \ref{th:9} using the same
arguments and references to get \begin{equation*}   i_{*}:K_{1}(L[\Gamma])\cong
K_{1}(M[\Gamma])^{\Delta}.
\end{equation*} The only difference is, that the elements of $\mathrm{Hom}$-groups in the proof are to be\textit{} totally positive on all symplectic representations.
\end{Rm}

\section{The case of rings of integers of ``local'' fields.}

\quad Let $G$ be a finite group. If $S$ is an integral domain of characteristic zero with field of
fractions $L$, then $\overline{L}$ will denote a chosen algebraic closure of $L$. We have a map
induced by base extension $\overline{L}\otimes_{S}-$
\begin{equation*}
\mathrm{Det}:\;K_{1}(S[G])\rightarrow
K_{1}(\overline{L}[G])=\underset{\chi}{\prod}\overline{L}^{\times}\cong \mathrm{Hom}(R_{G
},\overline{L}^{\times}),
\end{equation*}
where the direct product extends over the irreducible $\overline{L}$-valued characters of $G$.
Since the map from $K_{1}(L[G])$ to $K_{1}(\overline{L}[G])$ induced by $\overline{L}\otimes_{L}-$
is injective (see Proposition \ref{prop:3} (i)), the Det-map factorizes over $K_{1}(L[G])$. We
write $ SK_{1}(S[G])$ for $\mathrm{Ker}(\mathrm{Det})$, so that we have an exact sequence
\begin{equation}
1\rightarrow  SK_{1}(S[G])\rightarrow K_{1}(S[G])\rightarrow \mathrm{Det}(K_{1}(S[G]))\rightarrow 1. \label{eq:2}
\end{equation}

Therefore we shall consider the two parts of $K_{1}$, namely the Det-part and the $ SK_{1}$-part, separately.

\begin{Rm} \label{rm:17}
From the exact sequence \eqref{eq:2} and the fact, that $K_{1}$ commutes with direct limits (see \cite{R.}, Exercise 2.1.9), we deduce, that Det and $ SK_{1}$ also commute with direct limits.
\end{Rm}

\subsection{The $\mathrm{Det}$-part.}

\quad We keep the notation of the introduction. In this subsection let $S=\mathcal{O}_{L}$,
 where $L$ is either an arbitrary (possible infinite) Galois extension $L^{0}$ of finite absolute ramification index over $\mathbb{Q}_{p}$ or the completion of such $L^{0}$. Then $S$ is a Noetherian local ring, i.e. $S$ has a unique maximal left ideal, and $S[G]$ is semilocal, i.e. the quotient $S[G]/rad(S[G])$ of the ring by its Jacobson radical is left Artinian (see Proposition 20.6 in \cite{L.}). We have the
 following

\begin{Prop} \label{th:1}
Let $\Lambda$ be a semilocal ring (for example $S[G]$). The maps
\begin{equation*}
\Lambda^{\times}= GL_{1}(\Lambda)\hookrightarrow GL(\Lambda)\twoheadrightarrow K_{1}(\Lambda)
\end{equation*}
induce an equality
\begin{equation*}
\mathrm{Det}(\Lambda^{\times})=\mathrm{Det}( GL(\Lambda))=\mathrm{Det}(K_{1}(\Lambda)).
\end{equation*}
\end{Prop}
\begin{proof}
See Theorem 40.31 in \cite{CR. 2}.
\end{proof}

\begin{Cj} \label{cj:1}
Let $S=\mathcal{O}_{L}$ and $G$ be as above. Let $\Delta$ be an open subgroup of $Gal(L^{0}/\mathbb{Q}_{p})$ acting coefficientwise on $S[G]$, and hence on Det-groups. Then
\begin{equation*}
i_{*}:\mathrm{Det}(S^{\Delta}[G]^{\times})\cong\mathrm{Det}(S[G]^{\times})^{\Delta}.
\end{equation*}
\end{Cj}

The proof of Conjecture \ref{cj:1} proceeds in two steps. At   present  we can prove step 1 and
thus Conjecture \ref{cj:1} only under further assumptions on $S$ (see Theorem \ref{th:34}). We
first do the proof for finite extensions and completions of infinite extensions, since $S$ is
$p$-adically complete in these cases, and then we generalize the statement to infinite algebraic
extensions using direct limits (see Remark \ref{rm:5}).

\begin{Rm}
The map $i_{*}$ in the conjecture is always a monomorphism, as the following diagram commutes and respects the action of $\Delta$
\begin{equation*}
\begin{array}{ccc}
\mathrm{Det}(S^{\Delta}[G]^{\times}) & \hookrightarrow & K_{1}(L^{\Delta}[G])\\
\Big\downarrow{i_{*}} & & \Big\downarrow i_{*} \\
\mathrm{Det}(S[G]^{\times}) & \hookrightarrow & K_{1}(L[G]),\\
\end{array}
\end{equation*}
and the right hand side map is injective (see section 1).
\end{Rm}

\subsubsection{Step 1. The $p$-group case.}\label{axiomatic}
\quad Let $G$ be a finite $p$-group. For this step we generalize the ideas of \cite{CPT} to the case of an infinite group $\Delta$. We want to do the proof in full generality of \cite{CPT}, thus we have to assume, that $S$ is an arbitrary unitary ring satisfying the following conditions: \newline
(i) $S$ is an integral domain, which is torsion free as an abelian group, \newline
(ii) the natural map $S\rightarrow\underset{\leftarrow}{\lim}\;S/p^{n}S$ is an isomorphism, so that $S$ is $p$-adically complete,\newline
(iii) $S$ supports a lift of Frobenius, that is to say an endomorphism $F:\;S\rightarrow S$ with the property that for all $s\in S$
\begin{equation*}
F(s)\equiv s^{p}\;\mathrm{mod}\;pS.
\end{equation*}

\begin{Rm} \label{rm:1}
Proposition \ref{th:1} holds also for $S$ satisfying the conditions (i)-(iii) above and $G$ being a finite $p$-group.
\end{Rm}
\begin{proof}
See Theorem 1.2 in \cite{CPT}.
\end{proof}

Now we are ready to formulate the main theorem of step 1.

\begin{Th} \label{th:3}
Let $G$ be a finite $p$-group. Let $S$ be an arbitrary unitary ring satisfying the conditions (i)-(iii) and $\Delta$ be a group acting on $S$ by the ring automorphisms, such that $R=S^{\Delta}$ also satisfies the conditions (i)-(iii). We do not suppose, that the lift of Frobenius $F_{R}$ is compatible with the lift of Frobenius $F_{S}$, so that $F_{S}\mid_{R}$ need not equal $F_{R}$. Then we have the isomorphism
\begin{equation*}
i_{*}:\mathrm{Det}(R[G]^{\times})\cong \mathrm{Det}(S[G]^{\times})^{\Delta},
\end{equation*}
where $\Delta$ acts on Det-groups coefficientwise.
\end{Th}
\begin{proof}
Since $\Delta$ acts on $S$ by the ring automorphisms, we have the equality $(S^{\times})^{\Delta}=R^{\times}$ and for any finitely generated free $S$-module $M=\underset{i}{\bigoplus}S e_{i}$, on which $\Delta$ acts coefficientwise, $M^{\Delta}$ is the finitely generated free $R$-module $M^{\Delta}=\underset{i}{\bigoplus}R e_{i}$.

$S$ and $R$ both satisfy the Hypothesis in \cite{CPT}, thus the proof of the theorem is identical with the proof of Theorem 4.1 in \cite{CPT}. Note that this proof does not depend on the condition, whether $\Delta$ is a finite group or not; it only uses the equalities above.
\end{proof}

\begin{Rm} \label{rm:2}
In particular, Theorem \ref{th:3} holds for $S=\mathcal{O}_{L}$, where $L$ is either a finite
unramified extension $L^{0}$ of $\mathbb{Q}_{p}$ or the completion of an infinite unramified
extension $L^{0}$ of $\mathbb{Q}_{p}$ - in other words for the ring of Witt vectors $W(\kappa)$ of
any algebraic extension $\kappa$ of $\mathbb{F}_{p}$ -, and for $\Delta$ being an open subgroup of
$Gal(L^{0}/\mathbb{Q}_{p})$.
\end{Rm}

\subsubsection{Step 2. Reduction of the general case to the $p$-group case.}

\quad In contrast to the first step the second one  can be proved in the full generality of
Conjecture \ref{cj:1} for a given axiomatic setting. So let $S$ and $G$ be as in the conjecture
(for infinite Galois extensions see Remark \ref{rm:5}, so we assume, that $S$ is $p$-adically
complete). Let $\Delta$ be an open subgroup of $Gal(L^{0}/\mathbb{Q}_{p})$, such that
$R=S^{\Delta}$ is the ring of integers of a finite extension of $\mathbb{Q}_{p}$. Then $S$ is a
local, Noetherian, normal ring satisfying
\newline (i) $S$ is an integral domain, which is torsion free as an abelian group, \newline (ii)
the natural map $S\rightarrow\underset{\leftarrow}{\lim}\;S/p^{n}S$ is an isomorphism, so that $S$
is $p$-adically complete, \newline (iii)  $S$ supports a lift of Frobenius, that is to say an
endomorphism $F:\;S\rightarrow S$ with the property that for all $s\in S$
\begin{equation*}
F(s)\equiv s^{p}\;\mathrm{mod}\;\mathfrak{M},
\end{equation*}
where $\mathfrak{M}$ is the maximal ideal of $S$.

Note that with $S$ also  $R$ satisfies (i)-(iii).

Our reduction step is based on the following axioms to be satisfied for  every  finite
$p$-group $G$ ($S$, $R$, $\Delta$ being as above); actually they were essential ingredients of step
1 in the unramified case and  we will show that  Conjecture \ref{cj:1} is true for any class of
rings satisfying them:
\begin{enumerate}
\item There exists a homomorphism $\nu$ defined using the lift of Frobenius on $S$
\begin{equation*}
\nu: \mathrm{Det}(1+I(S[G]))\rightarrow L[\mathcal{C}_{G}],
\end{equation*}
such that $\mathcal{L}=\nu\circ \mathrm{Det}$ (for the definition of $\mathcal{L}$ see pp. 12-13 in \cite{CPT}). Here $I(S[G])$ denotes the augmentation ideal of the group ring $S[G]$ and $\mathcal{C}_{G}$ denotes the set of conjugacy classes of $G$.
\item Let $\nu'$ denote the restriction of the homomorphism $\nu$ to $\mathrm{Det}(1+\mathcal{A}(S[G]))$, where $\mathcal{A}(S[G])$ is the kernel of the natural map from $S[G]$ to $S[G^{ab}]$, then $\nu'$ is an isomorphism
\begin{equation*}
\mathrm{Det}(1+\mathcal{A}(S[G]))\;\tilde{\longrightarrow}\; p\phi(\mathcal{A}(S[G])),
\end{equation*}
and $\phi:L[G]\rightarrow L[\mathcal{C}_{G}]$ denotes the $L$-linear map obtained by mapping each element of $G$ to its conjugacy class.
\item We have the exact sequence
\begin{equation*}
0\rightarrow \phi(\mathcal{A}(S[G]))\xrightarrow{(\nu')^{-1}\circ (p\cdot)} \mathrm{Det}(S[G]^{\times})\rightarrow S[G^{ab}]^{\times}\rightarrow 1.
\end{equation*}
\item We have the isomorphism
\begin{equation*}
i_{*}:\mathrm{Det}(S^{\Delta}[G]^{\times})\cong \mathrm{Det}(S[G]^{\times})^{\Delta},
\end{equation*}
where $\Delta$ acts coefficientwise on Det-groups.
\end{enumerate}
From now we   assume these axioms  and describe  the reduction step in this hypothetical
generality, in the hope to prove Conjecture \ref{cj:1} some day in full generality.

Most ideas and techniques of the proof are contained in \cite{CPT}. Thus, we only have to relax the condition (iii) on the ring in \cite{CPT}. So, let $S$ be as above. The proof will now proceed in different stages, restricting $G$ first to special types of groups.

\begin{flushleft}
\textbf{$\mathbb{Q}$-$p$-elementary groups.}
\end{flushleft}

We begin with an algebraic result which we shall require later on this stage.

Suppose, that $\mathcal{O}_{K}$ is the ring of integers of a finite unramified extension $K$ of $\mathbb{Q}_{p}$. Let $A$ denote the ring $S\otimes_{\mathbb{Z}_{p}}\mathcal{O}_{K}$ and let $M$ be the ring of fractions of $A$. Since $M$ is a separable $L$-algebra, it can be written as a finite product of field extensions $M_{i}$ of $L$:
\begin{equation*}
M=\prod^{n}_{i=1}M_{i}.
\end{equation*}
Since $K$ is a finite unramified extension of $\mathbb{Q}_{p}$, we know that $A$ is etale over $S$ and hence is normal (see e.g. page 27 in \cite{Mi.}). If $A_{i}$ is the normalization of $S$ in $M_{i}$, then
\begin{equation*}
A=\prod^{n}_{i=1}A_{i}.
\end{equation*}
\begin{Le} \label{le:1}
Let $F$ denote the lift of Frobenius on $A$ given by the tensor product of the lift of Frobenius on $S$ with the Frobenius automorphism of $\mathcal{O}_{K}$; then $F(A_{i})\subset A_{i}$.
\end{Le}
\begin{proof}
Let $\left\{e_{i}\right\}$ denote the system of primitive orthogonal idempotents associated to the above product decomposition of $A$. As $F$ is a ${\mathbb{Z}_{p}}$-algebra endomorphism, we know, that $\left\{F(e_{i})\right\}$ is a system of orthogonal idempotents with
\begin{equation*}
1=\sum^{n}_{i=1}F(e_{i})
\end{equation*}
and so this system corresponds to a decomposition of the commutative algebra $A$ into $n$ components. Since the decomposition of Noetherian commutative algebras into indecomposable algebras is unique, we must have $F(e_{i})=e_{\pi(i)}$ for some permutation $\pi$ of $\left\{1,...,n\right\}$. It will suffice to show, that the permutation $\pi$ is the identity. Suppose for contradiction, that for some $i$, we have $\pi(i)=j\neq i$. We know by definition, that
\begin{equation*}
F(e_{i})\equiv e_{i}^{p}=e_{i}\;(\mathrm{mod}\;\mathfrak{M}A)
\end{equation*}
and so
\begin{equation*}
e_{i}\equiv F(e_{i})\cdot e_{i}\equiv e_{j}\cdot e_{i}=0\;(\mathrm{mod}\;\mathfrak{M}A).
\end{equation*}
However, by Theorem 6.7 (page 123) in \cite{CR. 1} we know, that, since $\mathfrak{M}A$ is contained in the radical of $A$, $e_{i}$ mod $\mathfrak{M}A$ must be a primitive idempotent of $A/\mathfrak{M}A$, and so we have the desired contradiction.
\end{proof}

Suppose $G$ is a $\mathbb{Q}$-$p$-elementary group, i.e. $G$ may be written as a semidirect product $C\rtimes P$, where $C$ is a cyclic normal subgroup of order s, which is prime to $p$, and where $P$ is a $p$-group. We decompose the commutative group ring $\mathbb{Z}_{p}[C]$ according as the divisors $m$ of $s$
\begin{equation} \label{eq:3}
\mathbb{Z}_{p}[C]=\prod_{m | s}\mathbb{Z}_{p}[m],
\end{equation}
where $\mathbb{Z}_{p}[m]$ is the semilocal ring
\begin{equation*}
\mathbb{Z}_{p}[m]=\mathbb{Z}[\zeta_{m}]\otimes_{\mathbb{Z}}\mathbb{Z}_{p},
\end{equation*}
and where $\zeta_{m}$ is a primitive $m$th root of unity in $\overline{\mathbb{Q}_{p}}$. For brevity we set $S[m]=S\otimes_{\mathbb{Z}_{p}}\mathbb{Z}_{p}[m]$, and we note, that by Lemma \ref{le:1} $S[m]$ decomposes as a product of rings satisfying (i)-(iii), where the Frobenius is given by the restriction of the tensor product of the lift of Frobenius on $S$ and the Frobenius automorphism of $\mathbb{Z}_{p}[m]$ to each component.

For each $m$ the conjugation action of $P$ on $C$ induces a homomorphism
\begin{equation*}
\alpha_{m}:P\rightarrow\mathrm{Aut}\left\langle \zeta_{m}\right\rangle
\end{equation*}
and we let $H_{m}=\mathrm{Ker}(\alpha_{m})$ and $A_{m}=\mathrm{Im}(\alpha_{m})$.

Tensoring the decomposition \eqref{eq:3} with $-\otimes_{\mathbb{Z}_{p}[C]}S[G]$ affords a
decomposition of $S$-algebras
\begin{equation} \label{eq:4}
S[G]=\prod_{m}S[m]\circ P,
\end{equation}
where $S[m]\circ P$ is the free $S[m]$-module on the set of elements of $P$ with the following multiplication $s_{1}p_{1}\cdot s_{2}p_{2}=s_{1}s_{2}^{p_{1}}p_{1}p_{2}$, $P$ acting on $S[m]$ through $A_{m}$. $S[m]\circ P$ is also called the twisted group ring. We shall study the determinant group $\mathrm{Det}( GL(S[G]))$ by studying the various subgroups $\mathrm{Det}( GL(S[m]\circ P))$. Note that the twisted group ring $S[m]\circ P$ contains the standard group ring $S[m][H_{m}]$. We therefore have the inclusion map $i:S[m][H_{m}]\rightarrow S[m]\circ P$. We also have a restriction map defined by choosing a transversal $\left\{a_{i}\right\}$ of $P/H_{m}$. This induces a restriction homomorphism
\begin{equation*}
\mathrm{res}: GL_{n}(S[m]\circ P)\rightarrow  GL_{n\mid A_{m}\mid}(S[m][H_{m}]).
\end{equation*}
By Proposition \ref{th:1} we know, that $\mathrm{Det}( GL(S[m][H_{m}]))=\mathrm{Det}(S[m][H_{m}]^{\times})$, and so we have defined the composition:
\begin{equation} \label{eq:5}
r_{m}:\mathrm{Det}( GL_{n}(S[m]\circ P))\rightarrow  GL_{n\mid A_{m}\mid}(S[m][H_{m}])\rightarrow \mathrm{Det}(S[m][H_{m}]^{\times}).
\end{equation}
Since for $\pi\in P$, $x\in (S[m]\circ P)^{\times}$, we know, that $\mathrm{Det}(\pi x \pi^{-1})=\mathrm{Det}(x)$, whence
\begin{equation*}
r_{m}:\mathrm{Det}( GL_{n}(S[m]\circ P))\rightarrow \mathrm{Det}(S[m][H_{m}]^{\times})^{A_{m}}.
\end{equation*}
Here $A_{m}$ acts via $\alpha_{m}$ on $S[m]$ and by conjugation on $H_{m}$. From (3.8) on page 69 in \cite{T.} we know, that $r_{m}$ is injective. Note for future reference, that for $x\in S[m][H_{m}]^{\times}$, $i(x)$ is mapped by restriction to the diagonal matrix $\mathrm{diag}(x^{a_{i}})$; thus we write $\mathrm{Det}(x)$ for the usual element of $\mathrm{Det}(S[m][H_{m}]^{\times})$ whereas $\mathrm{Det}(i(x))$ denotes an element of $\mathrm{Det}((S[m]\circ P)^{\times})$. These two determinants are related by the identity
\begin{equation*}
r_{m}(\mathrm{Det}(i(x)))=\prod_{a\in A_{m}}\mathrm{Det}(x^{a})=N_{A_{m}}(\mathrm{Det}(x)).
\end{equation*}
Next we describe $\mathrm{Det}((S[m]\circ P)^{\times})$, and more generally $\mathrm{Det}( GL(S[m]\circ P))$, and the maps $i$ and $r_{m}$ in terms of character functions. In Lemma \ref{le:2} we shall see, that every irreducible character of $G$ may be written in the form $\mathrm{Ind}^{G}_{H_{m}}(\phi_{m})$ for some $m$, where $\phi_{m}$ is an abelian character of $H_{m}$ with the property, that the restriction of $\phi_{m}$ to $C$ has order $m$. With this notation the elements $\mathrm{Det}(i(x))$ in $\mathrm{Det}((S[m]\circ P)^{\times})$ are character functions on such $\mathrm{Ind}^{G}_{H_{m}}(\phi_{m})$ with
\begin{equation*}
\mathrm{Det}(i(x))(\mathrm{Ind}^{G}_{H_{m}}(\phi_{m}))=r_{m}(\mathrm{Det}(i(x)))(\phi_{m})=\prod_{a\in A_{m}}\mathrm{Det}(x^{a})(\phi_{m}).
\end{equation*}
It is also instructive to see the above in the context of K-theory. We then have induction and restriction maps
\begin{equation*}
\begin{array}{ccccccc}
ind & : & K_{1}(S[m][H_{m}]) & \leftrightarrows & K_{1}(S[m]\circ P) & : & r_{m} \\
ind & : & \mathrm{Det}(S[m][H_{m}]^{\times}) & \leftrightarrows & \mathrm{Det}((S[m]\circ P)^{\times}) & : & r_{m}. \\
\end{array}
\end{equation*}
Similarly we have the corresponding maps on the representation rings
\begin{equation*}
\begin{array}{ccc}
\mathrm{K}_{0}(\overline{\mathbb{Q}_{p}}[m][H_{m}]) & \leftrightarrows & \mathrm{K}_{0}(\overline{\mathbb{Q}_{p}}[m]\circ P) \\
\end{array}
\end{equation*}
and by Mackey theory the induction map $i=\mathrm{Ind}^{P}_{H_{m}}$ maps the ring $\mathrm{K}_{0}(\overline{\mathbb{Q}_{p}}[m][H_{m}])$ onto $\mathrm{K}_{0}(\overline{\mathbb{Q}_{p}}[m]\circ P)$ (see page 68 in \cite{T.}).

Let $I_{H_{m}}$, resp. $I_{P}$, denote the augmentation ideal of $S[m][H_{m}]$, resp. the two sided $S[m]\circ P$-ideal generated by $I_{H_{m}}$. The main result on this stage is to show:
\begin{Th} \label{th:4}
The map $r_{m}$ defined in (6) gives an isomorphism
\begin{equation*}
r_{m}:\mathrm{Det}( GL(S[m]\circ P))=\mathrm{Det}((S[m]\circ P)^{\times})\rightarrow \mathrm{Det}(S[m][H_{m}]^{\times})^{A_{m}}.
\end{equation*}
\end{Th}
\begin{proof}
We have seen, that $r_{m}$ is injective on $\mathrm{Det}( GL(S[m]\circ P))$; we now show, that $r_{m}$ maps $\mathrm{Det}((S[m]\circ P)^{\times})$ onto $\mathrm{Det}(S[m][H_{m}]^{\times})^{A_{m}}$.

First we put $\widetilde{P}_{m}=P/[H_{m},H_{m}]$, then the ring $S[m]\circ \widetilde{P}_{m}$ is isomorphic to the ring of $| A_{m}|\times | A_{m} |$ matrices over $(S[m][H^{ab}_{m}])^{A_{m}}$. To see it, we note, that by Lemma \ref{le:1} $S[m]=\prod_{i}S_{i}$, where $S_{i}$ are complete local rings. Further, by Lemma 8.2 on page 613 in \cite{W.} $S_{i}\circ \widetilde{P}_{m}$ is an Azumaya algebra over some complete local ring $B_{i}$. Thus, it represents a class in the Brauer group of $B_{i}$, which is the quotient group of the group of Azumaya algebras by the subgroup of full matrix algebras. But by Theorem 6.5 in \cite{AG} the Brauer group of $B_{i}$ is isomorphic to the Brauer group of the residue class field $b_{i}$ of $B_{i}$, and thus is trivial, as $b_{i}$ is an algebraic extension of $\mathbb{F}_{p}$ and the Brauer group of a quasi-algebraically closed field is trivial (see Theorem 6.5.4, Theorem 6.5.7 and Proposition 6.5.8 in \cite{NSW}). Hence we see, that $r_{m}$ induces an isomorphism
\begin{equation} \label{eq:6}
\mathrm{Det}((S[m]\circ \widetilde{P}_{m})^{\times})\cong ((S[m][H^{ab}_{m}])^{A_{m}})^{\times}.
\end{equation}
From the axioms (2) and (3) above (which trivially extends to products of rings, since formation of determinants commutes with ring products) and using \eqref{eq:6}, we have a commutative diagram with exact top row:
\begin{equation} \label{eq:7}
\begin{array}{ccccc}
\mathrm{Det}(1+\mathcal{A}(S[m][H_{m}]))^{A_{m}} & \hookrightarrow & \mathrm{Det}(S[m][H_{m}]^{\times})^{A_{m}} & \rightarrow & ((S[m][H^{ab}_{m}])^{A_{m}})^{\times}\\
\Big\uparrow & & {\Big\uparrow}{r_{m}} & & {\Big\uparrow}{\wr}\\
\mathrm{Det}(i(1+\mathcal{A}(S[m][H_{m}]))) & \hookrightarrow & \mathrm{Det}((S[m]\circ P)^{\times}) & \twoheadrightarrow & \mathrm{Det}((S[m]\circ \widetilde{P}_{m})^{\times}).\\
\end{array}
\end{equation}
It will therefore suffice to show
\begin{equation*}
r_{m}(\mathrm{Det}(i(1+\mathcal{A}(S[m][H_{m}]))))\supseteq \mathrm{Det}(1+\mathcal{A}(S[m][H_{m}]))^{A_{m}}
\end{equation*}
and this follows from the commutative diagram
\begin{equation*}
\begin{array}{ccc}
\mathrm{Det}(1+\mathcal{A}(S[m][H_{m}])) & \overset{\nu}{\tilde{\longrightarrow}} & \phi(\mathcal{A}(S[H_{m}]))\otimes_{S}S[m]\\
{\Big\downarrow}{r_{m}} & & {\Big\downarrow}{\mathrm{tr}_{A_{m}}}\\
\mathrm{Det}(1+\mathcal{A}(S[m][H_{m}]))^{A_{m}} & \overset{\nu^{A_{m}}}{\tilde{\longrightarrow}} & (\phi(\mathcal{A}(S[H_{m}]))\otimes_{S}S[m])^{A_{m}}.\\
\end{array}
\end{equation*}
Recall that $F$ is the tensor product of the lift of Frobenius on $S$ with the Frobenius automorphism of $\mathbb{Q}_{p}(\zeta_{m})/\mathbb{Q}_{p}$. Note also that $A_{m}$ acts on $S[m]=S\otimes \mathbb{Z}_{p}[m]$ via the second factor; so, because $G$ is a $\mathbb{Q}$-$p$-elementary, the action of $A_{m}$ on $\left\langle\chi(G)\right\rangle$ factors through $Gal(\mathbb{Q}_{p}(\zeta_{m})/\mathbb{Q}_{p})$; this guarantees, that the actions of $F$ and $A_{m}$ commute; hence $\nu$ is an isomorphism of $A_{m}$-modules, and this gives the bottom row in the above diagram.

Since $S[m]$ is a free $S[A_{m}]$-module, it follows, that
$\phi(\mathcal{A}(S[H_{m}]))\otimes_{S}S[m]$ is a projective $S[A_{m}]$-module (with diagonal
action); and so $\mathrm{tr}_{A_{m}}$, and therefore $r_{m}$, is surjective.
\end{proof}

Finally we show:

\begin{Th} \label{th:5}
Let $G$ be a finite $\mathbb{Q}$-$p$-elementary group and let $S$, $\Delta$ and $R$ be as above. Further, let $\Delta$ act coefficientwise on $\mathrm{Det}(S[G]^{\times})$, then
\begin{equation*}
\mathrm{Det}(S[G]^{\times})^{\Delta}=\mathrm{Det}(R[G]^{\times}).
\end{equation*}
\end{Th}
\begin{proof}
By \eqref{eq:4} together with Theorem \ref{th:4}
\begin{equation*}
\mathrm{Det}(S[G]^{\times})^{\Delta}=\bigoplus_{m}\mathrm{Det}((S[m]\circ P)^{\times})^{\Delta}=\bigoplus_{m}(\mathrm{Det}(S[m][H_{m}]^{\times})^{A_{m}})^{\Delta}.
\end{equation*}
Recall that $\Delta$ acts via the first term in $S[m]\circ P=S\otimes_{R}(R[m]\circ P)$ and that $A_{m}$ acts via the second term; hence the actions of $\Delta$ and $A_{m}$ commute on $S[m][H_{m}]=(S\otimes_{R}R[m])[H_{m}]$; hence we see, that
\begin{equation*}
\mathrm{Det}(S[G]^{\times})^{\Delta}=\bigoplus_{m}(\mathrm{Det}(S[m][H_{m}]^{\times}))^{A_{m}\times \Delta}=\bigoplus_{m}(\mathrm{Det}(S[m][H_{m}]^{\times})^{\Delta})^{A_{m}}
\end{equation*}
and so by the fact 4 and Theorem \ref{th:4} together with \eqref{eq:4} we have equalities
\begin{equation*}
\mathrm{Det}(S[G]^{\times})^{\Delta}=\bigoplus_{m}(\mathrm{Det}(R[m][H_{m}]^{\times}))^{A_{m}}=\bigoplus_{m}\mathrm{Det}((R[m]\circ P)^{\times})=\mathrm{Det}(R[G]^{\times}).
\end{equation*}
\end{proof}

\textbf{Application.} We conclude this stage by considering the implications of the above result for an arbitrary finite group $G$. From 12.6 in \cite{S.} we know, that we can find an integer $l$ prime to $p$, $\mathbb{Q}$-$p$-elementary subgroups $H_{i}$ of $G$, integers $n_{i}$, and $\theta_{i}\in\mathrm{K}_{0}(\mathbb{Q}_{p}[H_{i}])$, such that
\begin{equation*}
l\cdot 1_{G}=\sum_{i} n_{i}\cdot \mathrm{Ind}^{G}_{H_{i}}(\theta_{i}).
\end{equation*}
Thus, given $\mathrm{Det}(x)\in\mathrm{Det}( GL(S[G]))^{\Delta}$, then by the Frobenius
structure of the module $ GL(S[G])$ over $\mathrm{K}_{0}(\mathbb{Q}_{p}[H_{i}])$ (see \S
5.a in \cite{CPT}) we have
\begin{equation*}
\mathrm{Det}(x)^{l}=\prod_{i}\big(\mathrm{Ind}^{G}_{H_{i}}(\theta_{i}\cdot\mathrm{Res}^{H_{i}}_{G}(\mathrm{Det}(x)))\big)^{n_{i}}.
\end{equation*}
However, Theorem \ref{th:5} above implies that
\begin{equation*}
\theta_{i}\cdot\mathrm{Res}^{H_{i}}_{G}(\mathrm{Det}(x))\in\mathrm{Det}(S[H_{i}]^{\times})^{\Delta}=\mathrm{Det}(R[H_{i}]^{\times}).
\end{equation*}
Thus we have shown
\begin{Th} \label{th:6}
For any finite group $G$ each element in the quotient group
\begin{equation*}
\mathrm{Det}( GL(S[G]))^{\Delta}/\mathrm{Det}( GL(R[G]))
\end{equation*}
has finite order, which is prime to $p$.
\end{Th}

\begin{flushleft}
\textbf{$\mathbb{Q}$-$l$-elementary groups.}
\end{flushleft}

We consider a prime $l\neq p$ and a $\mathbb{Q}$-$l$-elementary group $G$, i.e $G$ may be written as $(C\times C')\rtimes L$, where $C$ is a cyclic $p$-group, $C'$ is a cyclic group of order prime to $pl$ and $L$ is an $l$-group. On this stage we show:
\begin{Th} \label{th:7}
If $G$ is a $\mathbb{Q}$-$l$-elementary group, then
\[ \mathrm{Det}(S[G]^{\times})^{\Delta}=\mathrm{Det}(R[G]^{\times}). \]
\end{Th}

Then, reasoning as in the Application on the previous stage, we can immediately deduce
\begin{Th} \label{th:8}
For an arbitrary finite group $G$ each element in the quotient group
\begin{equation*}
\mathrm{Det}( GL(S[G]))^{\Delta}/\mathrm{Det}( GL(R[G]))
\end{equation*}
has finite order, which is prime to $l$.
\end{Th}

This, together with Theorem \ref{th:6} and Proposition \ref{th:1}, will then establish Conjecture \ref{cj:1}.

Prior to proving Theorem \ref{th:7}, we first need to recall four preparatory results:
\begin{Le} \label{le:2}
Each irreducible character $\chi$ of $G$ can be written in the form $\mathrm{Ind}^{G}_{\Omega}\phi$, where $\phi$ is an abelian character of a subgroup $\Omega$, which contains $C\times C'$.
\end{Le}
\begin{proof}
See 8.2 in \cite{S.}
\end{proof}

\begin{Prop} \label{prop:1}
Let $\mathcal{O}$ denote the ring of integers of the finite extension of $\mathbb{Q}_{p}$ generated by the values of all characters of $G$, let $\mathfrak{m}$ denote the maximal ideal of $\mathcal{O}$, and let $\mathcal{P}$ denote the $S\otimes_{\mathbb{Z}_{p}}\mathcal{O}$-ideal generated by $\mathfrak{m}$. With the notation of the previous lemma, we write $\phi=\phi' \phi_{p}$, where $\phi'$ (resp. $\phi_{p}$) has order prime to $p$ (resp. $p$-power order); and we put $\chi'=\mathrm{Ind}^{G}_{\Omega}\phi'$. Then for $r\in  GL(S[G])$ we have the congruence
\begin{equation*}
\mathrm{Det}(r)(\chi-\chi')\equiv 1\;(\mathrm{mod}\;\mathcal{P}).
\end{equation*}
\end{Prop}
\begin{proof}
It is an easy generalization of Lemma 1.3 on page 35 in \cite{T.}.
\end{proof}

\begin{Prop} \label{prop:2}
Put $G'=G/C$. Then $\mathbb{Z}_{p}[G']$ is a split maximal $\mathbb{Z}_{p}$-order, i.e. it is a product of matrix rings
\begin{equation*}
\mathbb{Z}_{p}[G']=\prod_{i}\mathrm{M}_{n_{i}}(\mathcal{O}_{i})
\end{equation*}
over (local) rings of integers $\mathcal{O}_{i}$. Thus we have the equalities
\begin{equation*}
\mathrm{Det}( GL(S[G']))=\prod_{i}(S\otimes_{\mathbb{Z}_{p}}\mathcal{O}_{i})^{\times}=\mathrm{Det}(S[G']^{\times})
\end{equation*}
and
\begin{equation*}
\mathrm{Det}(S[G']^{\times})^{\Delta}=\prod_{i}\mathrm{Det}((S\otimes\mathcal{O}_{i})^{\times})^{\Delta}=\prod_{i}(S\otimes\mathcal{O}_{i})^{\times \Delta}=
\end{equation*}
\begin{equation*}
\prod_{i}(R\otimes\mathcal{O}_{i})^{\times}=\prod_{i}\mathrm{Det}((R\otimes\mathcal{O}_{i})^{\times})=\mathrm{Det}(R[G']^{\times}).
\end{equation*}
\end{Prop}
\begin{proof}
See Theorem 41.1 and Theorem 41.7 in \cite{Re.}.
\end{proof}

\begin{Le} \label{le:3}
With the previous notation, the map $S[G]^{\times}\rightarrow S[G']^{\times}$ is surjective.
\end{Le}

\begin{proof}
Consider the canonical projection
\begin{equation*}
A=S[G]\stackrel{\varphi}{\twoheadrightarrow}S[G']=B.
\end{equation*}

Firstly, we have to prove, that $\mathrm{Ker}(\varphi)$ is contained in the radical of $A$. We know, that $\mathrm{Ker}(\varphi)=I(C)\subset A$, where $I(C)$ is generated by $1-\sigma,\;\sigma\in C$. Thus there is a positive integer $r$, such that $(\mathrm{Ker}(\varphi))^{p^{r}}\subset\mathfrak{M}$, where $\mathfrak{M}$ is the maximal ideal of $S$. A positive power of an ideal is contained in the radical of a ring if and only if the ideal is contained in the radical (the image of such ideal would be a nilpotent ideal of $A/rad(A)$, hence it is contained in $rad(A/rad(A))=(0)$, see \cite{Re.}, {\S}6). Hence it is enough to prove, that $\mathfrak{M}$ is contained in $rad(A)$. But this follows from the more general lemma below (see Lemma \ref{le:15}), as the radical of the ring $A$ is the intersection of all maximal left ideals of $A$ (see Theorem 6.3 in \cite{Re.}).

Secondly, by the Theorem 6.10 in \cite{Re.} and the first step of the proof $\varphi$ induces an isomorphism
\begin{equation*}
\bar{\varphi}: A/rad(A)\tilde{\rightarrow} B/rad(B),
\end{equation*}
as $\varphi(rad(A))\subset rad(B)$.

Finally, let $x\in B^{\times}$. We denote its image in $(B/rad(B))^{\times}$ by $\bar{x}$. Since $\varphi$ is surjective, we can lift $x$ to an element $y\in A$. The image of $y$ in $A/rad(A)$ denoted $\bar{y}$ is then mapped under $\bar{\varphi}$ onto $\bar{x}$, thus it is contained in $(A/rad(A))^{\times}$ and hence $y\in A^{\times}$, as $1+rad(A)\subset A^{\times}$ (see Theorem 6.5 in \cite{Re.}).
\end{proof}

\begin{Le} \label{le:15}
Let $S$ be a discrete valuation ring, but not a field. Let $\mathfrak{M}$ denote its unique maximal (left) ideal. Assume, that we have an inclusion of the rings $i:S\rightarrow A$, where $A$ is an arbitrary ring, then $i(\mathfrak{M})\cdot S$ is contained in every maximal non-zero left ideal $\mathcal{M}$ of $A$.
\end{Le}
\begin{proof}
The preimage $i^{-1}(\mathcal{M})$ of $\mathcal{M}$ is a non-zero prime left ideal of $S$, thus it is a maximal left ideal, hence it coincides with $\mathfrak{M}$. Now applying $i$ we get $i(\mathfrak{M})\cdot S=i(i^{-1}(\mathcal{M}))\cdot S\subseteq \mathcal{M}$.
\end{proof}

\begin{proof}[{Proof of Theorem \ref{th:7}.}] Suppose, that we are given $\mathrm{Det}(z)\in
\mathrm{Det}(S[G]^{\times})^{\Delta}$ and let $z'$ denote the image of $z$ in $S[G']$. Then by Proposition
\ref{prop:2} we know, that we can find $x'\in R[G']^{\times}$ with $\mathrm{Det}(x')=\mathrm{Det}(z')$;
moreover, by Lemma \ref{le:3} we can find $x\in R[G]^{\times}$ with image $x'$ in $R[G']^{\times}$. Thus, to
conclude, it will be sufficient to show, that $\mathrm{Det}(zx^{-1})$ is in $\mathrm{Det}(R[G]^{\times})$.
However, by construction, $\mathrm{Det}(zx^{-1})$ is trivial on characters inflated from $G'$, and
so by Proposition \ref{prop:1} we see that
\begin{equation*}
\mathrm{Det}(zx^{-1})(\chi)=\mathrm{Det}(zx^{-1})(\chi-\chi')\equiv 1\;(\mathrm{mod}\;\mathcal{P}),\;\;\text{for all }\chi.
\end{equation*}
Hence $\mathrm{Det}(zx^{-1})$ is in $\mathrm{Hom}(R_{G},1+\mathcal{P})$, where $R_{G}$ denote the
group of virtual $\overline{L}$-valued characters of $G$. Moreover, since $\mathrm{Det}(zx^{-1})$
is invariant under the action of $\Delta$, $\mathrm{Det}(zx^{-1})\in
\mathrm{Hom}_{\Delta}(R_{G},1+\mathcal{P})$. In the case, where $R$ is the ring of integers of a
finite extension of $\mathbb{Q}_{p}$, the last Hom-group is isomorphic to the finite product of
groups of higher principal units $U^{(1)}$ of finite extensions of $\mathbb{Q}_{p}$, and hence is a
pro-$p$-group. Thus $\mathrm{Det}(zx^{-1})$ is a pro-$p$-element of
$\mathrm{Det}(S[G]^{\times})^{\Delta}$. But by Theorem \ref{th:6} we know, that $\mathrm{Det}(zx^{-1})$
has image in the quotient group $\mathrm{Det}(S[G]^{\times})^{\Delta}/\mathrm{Det}(R[G]^{\times})$ of finite
order, which is prime to $p$. Therefore we may deduce, that $\mathrm{Det}(zx^{-1})$ is in
$\mathrm{Det}(R[G]^{\times})$.   \end{proof}

\subsubsection{Conjecture \ref{cj:1}: proved cases and one generalization.}

\quad Because of our restrictions in the $p$-group case we have proved Conjecture \ref{cj:1} only
for $S=W(\kappa)$ the ring of Witt vectors of an algebraic extension $\kappa$ of $\mathbb{F}_p$.
There are some possible generalizations of this result as will be explained now.

\begin{Rm} \label{rm:4}
Using the reduction step described on the page 92 in \cite{T.}, we can prove Conjecture \ref{cj:1}
for $S=\mathcal{O}_{L}$, where $L$ is the completion of a tamely ramified extension $L^{0}$ of
finite absolute ramification index over $\mathbb{Q}_{p}$, i.e., finite tamely ramified extension of
the ring of Witt vectors of an algebraic extension $\kappa$ of $\mathbb{F}_p$, and $\Delta$ being
an open subgroup of $Gal(L^{0}/\mathbb{Q}_{p})$ containing the inertia group.
\end{Rm}

\begin{Rm} \label{rm:5}
The infinite Galois extensions can always be written as a direct limit (or simply union) of their
finite subextensions, and so their rings of integers, too. Thus, if we have proved Conjecture
\ref{cj:1} for some class of finite Galois extensions of $\mathbb{Q}_{p}$, we can obtain it also
for infinite extensions in the same ind-class, as the Det-map commutes with direct limits (see
Remark \ref{rm:17}).

Explicitly, let $S$ be the ring of integers of an infinite extension $L$ of $\mathbb{Q}_{p}$. Let $L=\underset{i}{\bigcup}L_{i}$ and $S=\underset{i}{\bigcup}S_{i}$, where $L_{i}$ are finite extensions and $S_{i}$ their rings of integers, and we have the statement of Conjecture \ref{cj:1} for all $S_{i}$. Further, let $\Delta$ be an open subgroup of $Gal(L/\mathbb{Q}_{p})$, then
\begin{equation*}
\mathrm{Det}(S[G]^{\times})^{\Delta}=\mathrm{Det}\bigg(\underset{i}{\bigcup}S_{i}[G]^{\times}\bigg)^{\Delta}=\bigg(\underset{i}{\bigcup}\mathrm{Det}(S_{i}[G]^{\times})\bigg)^{\Delta}=\underset{i}{\bigcup}\bigg(\mathrm{Det}(S_{i}[G]^{\times})^{\Delta}\bigg)=
\end{equation*}
\begin{equation*}
\underset{i}{\bigcup}\bigg(\mathrm{Det}(S_{i}^{\Delta}[G]^{\times})\bigg)=\mathrm{Det}\bigg(\underset{i}{\bigcup}S_{i}^{\Delta}[G]^{\times}\bigg)=\mathrm{Det}(S^{\Delta}[G]^{\times}),
\end{equation*}
where $\Delta$ acts on $S_{i}$ through the corresponding quotient group and the union commutes with such defined $\Delta$-action. The maps between $\mathrm{Det}$-groups induced by inclusions of rings are inclusions by Theorem 13 in \cite{Fr.}.

For example, Conjecture \ref{cj:1} holds for $S=\mathcal{O}_{L}$, where $L$ is the maximal
unramified extension of $\mathbb{Q}_{p}$ and $\Delta$ is an open subgroup of
$Gal(L/\mathbb{Q}_{p})$ or using the previous remark $L$ is the maximal tamely ramified extension
of $\mathbb{Q}_{p}$ and $\Delta$ is an open subgroup of $Gal(L/\mathbb{Q}_{p})$ containing the
inertia group.
\end{Rm}

Remarks \ref{rm:2}, \ref{rm:4} and \ref{rm:5} imply

\begin{Th} \label{th:34}
Let $G$ be a finite group. Let $S=\mathcal{O}_{L}$, where $L$ is either an arbitrary (possibly
infinite) tamely ramified extension $L^{0}$ of $\mathbb{Q}_{p}$ (type 1) or the completion of a
tamely ramified extension $L^{0}$ of finite absolute ramification index over $\mathbb{Q}_{p}$ (type
2), and let $\Delta$ be an open subgroup of $Gal(L^{0}/\mathbb{Q}_{p})$ containing the inertia
group. Then
\begin{equation*}
i_{*}:\mathrm{Det}(S^{\Delta}[G]^{\times})\cong\mathrm{Det}(S[G]^{\times})^{\Delta}.
\end{equation*}
\end{Th}

We conclude this subsection with the following result generalising Theorem \ref{th:34}  to the case
of compact p-adic Lie groups and their Iwasawa algebras.   Let $\mathcal{G}$ be a compact p-adic
Lie group,  let $S=\mathcal{O}_{L}$ be as in the theorem and denote by $R=S^\Delta$ the ring of
integers of a finite unramified extension $K$ over $\mathbb{Q}_{p}$. We denote   by $\Delta$ the
Galois group $Gal(L^{0}/K)$ and write
\begin{equation*}
\Lambda_{S}(\mathcal{G}):=S[[\mathcal{G}]] \;\;\; (\text{resp.
}\Lambda_{R}(\mathcal{G}):=R[[\mathcal{G}]])
\end{equation*}
for the Iwasawa algebra of $\mathcal{G}$ with coefficients in $S$ (resp. in $R$). Note that it is a
Noetherian pseudocompact ring (resp. a Noetherian compact ring). For the notion of pseudocompact rings and algebras see \cite{Br.}. It is well known, that
$\Lambda_{R}(\mathcal{G})$ is a semilocal ring (see Proposition 5.2.16 in \cite{NSW}), using the
idea in the proof of Lemma \ref{le:3} we deduce, that the same is true for
$\Lambda_{S}(\mathcal{G})$.

In the following we will use Froehlich's Hom-description as it has been adapted to Iwasawa
theory by Ritter and Weiss in \cite{RW}. We have the following commutative diagram
\begin{equation*}
\begin{array}{ccc}
 K_{1}(\Lambda_{S}(\mathcal{G})) & \xrightarrow{\mathrm{Det}} & \mathrm{Hom}_{G_{L^{0}}}(R_\mathcal{G},\mathcal{O}^{\times}_{\mathbb{C}_{p}}) \\
\Big\uparrow & & \bigcup \\
K_{1}(\Lambda_{R}(\mathcal{G})) & \xrightarrow{\mathrm{Det}} & \mathrm{Hom}_{G_{K}}(R_\mathcal{G},\mathcal{O}^{\times}_{\mathbb{C}_{p}}),\\
\end{array}
\end{equation*}
where $G_{L^{0}}=Gal(\overline{L}/L^{0})$, $G_{L}=Gal(\overline{L}/L)$ and $R_\mathcal{G}$ as
before is the free abelian group on the isomorphism classes of irreducible
$\overline{\mathbb{Q}_p}$-valued Artin representations of $\mathcal{G}$.

Now we are ready to formulate

\begin{Th} \label{th:11}
With the notation as above we have
\begin{equation*}
\mathrm{Det}(K_{1}(\Lambda_{S}(\mathcal{G})))^{\Delta}=\mathrm{Det}(K_{1}(\Lambda_{R}(\mathcal{G}))),
\end{equation*}
where $\Delta$ acts on the $K_{1}$-groups coefficientwise.
\end{Th}
\begin{proof}
From the diagram
\begin{equation*}
\begin{array}{ccc}
\Lambda_{S}(\mathcal{G})^{\times} & \xrightarrow{\mathrm{Det}} & \mathrm{Hom}_{G_{L^{0}}}(R_\mathcal{G},\mathcal{O}^{\times}_{\mathbb{C}_{p}}) \\
\bigcup & & \bigcup \\
\Lambda_{R}(\mathcal{G})^{\times} & \xrightarrow{\mathrm{Det}} & \mathrm{Hom}_{G_{K}}(R_\mathcal{G},\mathcal{O}^{\times}_{\mathbb{C}_{p}}),\\
\end{array}
\end{equation*}
which is commutative by the construction, and from Proposition \ref{th:1} we get the first obvious inclusion
\begin{equation*}
\mathrm{Det}(\Lambda_{R}(\mathcal{G})^{\times})=\mathrm{Det}(K_{1}(\Lambda_{R}(\mathcal{G})))\subseteq\mathrm{Det}(K_{1}(\Lambda_{S}(\mathcal{G})))^{\Delta}=\mathrm{Det}(\Lambda_{S}(\mathcal{G})^{\times})^{\Delta}.
\end{equation*}

For the opposite inclusion we use Theorem \ref{th:34}, then we only have to show, how the general case can be reduced to the case
of finite groups. To this end write $\mathcal{G}=\underset{n}{\varprojlim}\;G_{n}$ as inverse limit of finite groups. By Theorem \ref{th:34} we have compatible continuous maps
\begin{equation*}
R[G_{n}]^{\times}\overset{\mathrm{Det}}{\twoheadrightarrow}\mathrm{Det}(K_{1}(\Lambda_{S}(G_{n})))^{\Delta}\hookrightarrow
\mathrm{Hom}_{G_{L^{0}}}(R_{G_{n}},\mathcal{O}^{\times}_{\mathbb{C}_{p}})^{\Delta},
\end{equation*}
where the topology on $\mathrm{Hom}_{G_{L^{0}}}(R_{G_{n}},\mathcal{O}^{\times}_{\mathbb{C}_{p}})$ is
induced from the valuation topology on $\mathbb{C}_{p}$. Taking the inverse limit yields, by the
compactness of $\Lambda_{R}(\mathcal{G})^{\times}=\underset{n}{\varprojlim}(R/\pi^{n}[G_{n}])^{\times}$ and by
letting $R_\mathcal{G}=\underset{n}{\varinjlim}\;R_{G_{n}}$, a factorization of the homomorphism
Det into
\begin{equation*}
\Lambda_{R}(\mathcal{G})^{\times}
\overset{\mathrm{Det}}{\twoheadrightarrow}(\underset{n}{\varprojlim}\;\mathrm{Det}(K_{1}(\Lambda_{S}(G_{n}))))^{\Delta}\hookrightarrow
\mathrm{Hom}_{G_{K}}(R_\mathcal{G},\mathcal{O}^{\times}_{\mathbb{C}_{p}}).
\end{equation*}
The claim follows, because denoting by
\begin{equation*}
res_{n}:\mathrm{Hom}_{G_{L^{0}}}(R_\mathcal{G},\mathcal{O}^{\times}_{\mathbb{C}_{p}})\rightarrow
\mathrm{Hom}_{G_{L^{0}}}(R_{G_{n}},\mathcal{O}^{\times}_{\mathbb{C}_{p}})
\end{equation*}
the restriction we obtain from the universal mapping property for
\begin{equation*}
\underset{n}{\varprojlim}\;\mathrm{Hom}_{G_{L^{0}}}(R_{G_{n}},\mathcal{O}^{\times}_{\mathbb{C}_{p}})\cong
\mathrm{Hom}_{G_{L^{0}}}(R_\mathcal{G},\mathcal{O}^{\times}_{\mathbb{C}_{p}})
\end{equation*}
inclusions
\begin{equation*}
\mathrm{Det}(K_{1}(\Lambda_{S}(\mathcal{G})))\subseteq
\underset{n}{\varprojlim}\;\mathrm{Im}(res_{n}\circ\mathrm{Det})\subseteq
\underset{n}{\varprojlim}\;\mathrm{Det}(K_{1}(\Lambda_{S}(G_{n}))),
\end{equation*}
whence
\begin{equation*}
\mathrm{Det}(K_{1}(\Lambda_{S}(\mathcal{G})))^{\Delta}\subseteq
\mathrm{Det}(\Lambda_{R}(\mathcal{G})^{\times})=\mathrm{Det}(K_{1}(\Lambda_{R}(\mathcal{G}))).
\end{equation*}
\end{proof}

For  applications of the theorem above in number theory  see \cite{BV}.

\begin{flushleft}
\textbf{End of the $\mathrm{Det}$-part.}
\end{flushleft}

\subsection{The $ SK_{1}$-part.}

\quad From \cite{O.} we have the following

\begin{Th} \label{th:15}
Let $R$ be the ring of integers in any finite extension $K$ of $\mathbb{Q}_{p}$. Then for any $p$-group $G$, there is an isomorphism
\begin{equation*}
\Theta_{RG}:  SK_{1}(R[G])\tilde{\rightarrow} \mathrm{H}_{2}(G)/\mathrm{H}^{ab}_{2}(G),
\end{equation*}
where $\mathrm{H}^{ab}_{2}(G)=\mathrm{Im}\big[\sum\left\{\mathrm{H}_{2}(H):\;H\subseteq G,\;H\text{ abelian}\right\}\xrightarrow{\sum\mathrm{Ind}}\mathrm{H}_{2}(G)\big]$

If $L\supseteq K$ is a finite extension, and if $S\subseteq L$ is the ring of integers , then \newline
(i) $i_{*}:  SK_{1}(R[G])\rightarrow  SK_{1}(S[G])$ (induced by inclusion) is an isomorphism, if $L/K$ is totally ramified; and\newline
(ii) $trf:  SK_{1}(S[G])\rightarrow  SK_{1}(R[G])$ (the transfer) is an isomorphism, if $L/K$ is unramified.
\end{Th}
\begin{proof}
See proof of Theorem 8.7 in \cite{O.}.
\end{proof}

We see, that $ SK_{1}(S[G])$ is independent of $S$. Note also, that $i_{*}$ and $trf$ are Galois
homomorphisms, hence $ SK_{1}(S[G])$ has trivial Galois action. In order to treat infinite
algebraic extensions of $\mathbb{Q}_{p}$ and the analogous descent statement of the introduction
for $SK_{1}$-groups we have to describe the maps $i_{*}$ induced by inclusions, as they appear in
the direct limits (see Remark \ref{rm:5} and Remark \ref{rm:17}).

Now we assume  that $K$ is unramified over $\mathbb{Q}_{p}$, then from Proposition 21 (i) in
\cite{O. 2}, which is also valid for $i_{*}$ with the same argument, we have commutative squares

\begin{equation*}
\begin{array}{ccc}
 SK_{1}(R[G]) & \xrightarrow{\Theta_{RG}} & \mathrm{H}_{2}(G)/\mathrm{H}^{ab}_{2}(G) \\
\Big\downarrow{i_{*}} & & \Big\downarrow{?} \\
 SK_{1}(S[G]) & \xrightarrow{\Theta_{SG}} & \mathrm{H}_{2}(G)/\mathrm{H}^{ab}_{2}(G),\\
\end{array}
\end{equation*}
and
\begin{equation*}
\begin{array}{ccc}
 SK_{1}(S[G]) & \xrightarrow{\Theta_{SG}} & \mathrm{H}_{2}(G)/\mathrm{H}^{ab}_{2}(G)\\
\Big\downarrow{trf} & & \Big\downarrow{id} \\
 SK_{1}(R[G]) & \xrightarrow{\Theta_{RG}} & \mathrm{H}_{2}(G)/\mathrm{H}^{ab}_{2}(G) \\
\end{array}
\end{equation*}
where $\Theta_{RG}$, $\Theta_{SG}$ and $trf$ are isomorphisms. To describe $i_{*}$ and $?$ we need
the following

\begin{Le} \label{le:17}
With the previous notation the map
\begin{equation*}
trf\circ i_{*}: K_{1}(R[G])\rightarrow K_{1}(S[G]) \rightarrow K_{1}(R[G])
\end{equation*}
is multiplication by $n=[L:K]$.
\end{Le}
\begin{proof}
By Proposition 1.18 in \cite{O.} the composite $trf\circ i_{*}$ is induced by tensoring over $R[G]$ with $S[G]\cong S\otimes_{R}R[G]$ regarded as an $(R[G],R[G])$-bimodule, where the bimodule structure on $S\otimes_{R}R[G]$ is given through the second factor in the natural way. Since $S$ is a free $R$-modules of rank $n$, $S\otimes_{R}R[G]\cong R[G]^{n}$ as $(R[G],R[G])$-bimodules, and so $trf\circ i_{*}$ is multiplication by $n$ on $K_{1}(R[G])$ (written additively).
\end{proof}

The map $trf\circ i_{*}$ on $K_{1}(R[G])$ corresponds via $\Theta_{RG}$ to the map
$\mathrm{H}_{2}(G)/\mathrm{H}^{ab}_{2}(G)\xrightarrow{n\cdot}\mathrm{H}_{2}(G)/\mathrm{H}^{ab}_{2}(G)$,
and since $trf$ corresponds to the identity map, $i_{*}$ also corresponds to the multiplication by
$n$. From the Theorem 3.14 in \cite{O.} we know, that $ SK_{1}(R[G])$ (hence
$\mathrm{H}_{2}(G)/\mathrm{H}^{ab}_{2}(G)$) is a finite $p$-group, so that we have proved the

\begin{Th} \label{th:21}
With the notation as above $i_{*}$ is an isomorphism in the following two cases \newline
(i) if $L/K$ is totally ramified, \newline
(ii) if $L/K$ is unramified and $p\nmid n$. In this case $i_{*}$ corresponds via $\Theta_{RG}$ to the multiplication by $n$ on the finite $p$-group (written additively).

If $L/K$ is unramified and $p | n$, then $i_{*}$ still corresponds via $\Theta_{RG}$ to the multiplication by $n$ on the finite $p$-group, which is neither surjective nor injective, as finite $p$-groups always have $p$-torsion elements.
\end{Th}

\begin{Cr} \label{cr:1}
The groups $ SK_{1}(R[G])$ and $ SK_{1}(S[G])$ are always isomorphic (as abstract groups with the
(trivial) action of $Gal(L/K)$), but the statement of the introduction for $SK_{1}$-groups, i.e.
\begin{equation} \label{eq:42}
i_{*}: SK_{1}(R[G])\cong  SK_{1}(S[G])^{\Delta}\;\;\;(\Delta=Gal(L/K)),
\end{equation}
holds only in the cases (i) and (ii) of Theorem \ref{th:21}.
\end{Cr}

\begin{Cr} \label{cr:2}
Let $M$ be an infinite algebraic extension of $\mathbb{Q}_{p}$ and let $M^{0}$ be the maximal
unramified extension of $\mathbb{Q}_{p}$ contained in $M$. We write $M$ as the direct limit (union)
of its finite subextensions and use Remark \ref{rm:17} and the theorem above to get the following
result:\newline If $\;p^{\infty}$ divides $[M^{0}:\mathbb{Q}_{p}]$ (as supernatural numbers), then
$ SK_{1}(\mathcal{O}_{M}[G])=1$ for every finite $p$-group $G$.
\end{Cr}

\begin{Rm} \label{rm:20}
From Corollary \ref{cr:2} we can obtain the generalization of Corollary \ref{cr:1} for infinite
extensions: Let $L$ be an infinite algebraic extension of $\mathbb{Q}_{p}$ and let $K=L^{\Delta}$,
where $\Delta$ is an open subgroup of $Gal(L/\mathbb{Q}_{p})$. Then the statement \eqref{eq:42}
 holds only in the cases (i) and (ii) of Theorem \ref{th:21}, here $p\nmid n$ as
supernatural numbers. In the case, where $\;p^{\infty}$ divides $[L^{0}:\mathbb{Q}_{p}]$ and $ SK_{1}(R[G])\neq
1$, $ SK_{1}(R[G])$ and $ SK_{1}(S[G])$ are not isomorphic even as abstract groups. See
\cite[Example 8.11]{O.} for an example of non-trivial $SK_{1}(\mathbb{Z}_p[G]).$
\end{Rm}

Now we generalize our results  to the case of an arbitrary finite group $G$. Note, that Theorem
3.14 in \cite{O.} and Lemma \ref{le:17} still hold in this case. From \cite{O. 3} we have the

\begin{Th} \label{th:22}
Let $R$ be the ring of integers in any finite extension $K$ of $\mathbb{Q}_{p}$ and let $G$ be a finite group. Let $L\supseteq K$ be a finite extension, and let $S\subseteq L$ be the ring of integers , then \newline
(ii) $i_{*}:  SK_{1}(R[G])\rightarrow  SK_{1}(S[G])$ is an isomorphism, if $L/K$ is totally ramified; \newline
(iii) $trf:  SK_{1}(S[G])\rightarrow  SK_{1}(R[G])$ is onto, if $L/K$ is unramified.
\end{Th}
\begin{proof}
See proof of Theorem 1 in \cite{O. 3}.
\end{proof}

We need some more notation. For any finite group $G$, and any fixed prime $p$, $G_{r}$ will denote
the set of $p$-regular elements in $G$, i.e., elements of order prime to $p$.
$\mathrm{H}_{n}(G,R(G_{r}))$ denotes the homology group induced by the conjugation action of $G$ on
$R(G_{r})$.

When $R$ is the ring of integers  in a finite unramified extension of $\mathbb{Q}_{p}$, then $\Phi$ denotes the automorphism of $\mathrm{H}_{n}(G,R(G_{r}))$ induced by the map $\Phi(\sum_{i}r_{i}g_{i})=\sum_{i}\varphi(r_{i})g^{p}_{i}$ on coefficients. We set $\mathrm{H}_{n}(G,R(G_{r}))_{\Phi}=\mathrm{H}_{n}(G,R(G_{r}))/(1-\Phi)$. In analogy with the $p$-group case, we define
\begin{equation*}
\mathrm{H}^{ab}_{2}(G,R(G_{r}))_{\Phi}=\mathrm{Im}\big[\sum\left\{\mathrm{H}_{2}(H,R(H_{r})):\;H\subseteq G,\;H\text{ abelian}\right\}\xrightarrow{\sum\mathrm{Ind}}
\end{equation*}
\begin{equation*}
\xrightarrow{\sum\mathrm{Ind}}\mathrm{H}_{2}(G,R(G_{r}))_{\Phi}\big].
\end{equation*}

We use the notation above to formulate the

\begin{Th} \label{th:23}
Let $R$ be the ring of integers  in a finite unramified extension of $\mathbb{Q}_{p}$. Then, for
any finite group $G$, there is an isomorphism
\begin{equation*}
\Theta_{G}:
SK_{1}(R[G])\tilde{\rightarrow}(R/(1-\varphi)R)\otimes_{\mathbb{Z}_{p}}\mathrm{H}_{2}(G,\mathbb{Z}_{p}(G_{r}))_{\Phi}/\mathrm{H}^{ab}_{2}(G,\mathbb{Z}_{p}(G_{r}))_{\Phi}.
\end{equation*}
\end{Th}

This new tensor product decomposition of $SK_1(R[G])$ in terms of $R/(1-\varphi)R$  and group
cohomology comes from   the results announced in \cite{CPT. 1}  and we are very grateful to T.
Chinburg, G. Pappas and M. J. Taylor for sharing this insight with us, which not least also
influenced our results below.

\begin{proof}
See proof of Theorem 12.10 in \cite{O.} and use the facts
\begin{equation*}
\mathrm{H}_{2}(G,R(G_{r}))_{\Phi}\cong
(R/(1-\varphi)R)\otimes_{\mathbb{Z}_{p}}\mathrm{H}_{2}(G,\mathbb{Z}_{p}(G_{r}))_{\Phi}
\end{equation*}
and
\begin{equation*}
\mathrm{H}^{ab}_{2}(G,R(G_{r}))_{\Phi}\cong
(R/(1-\varphi)R)\otimes_{\mathbb{Z}_{p}}\mathrm{H}^{ab}_{2}(G,\mathbb{Z}_{p}(G_{r}))_{\Phi}.
\end{equation*}
\end{proof}

Arguing as in the $p$-group case we deduce the following

\begin{Th} \label{th:24}
With the notation as above $i_{*}$ is \newline
(i) an isomorphism, if $L/K$ is totally ramified; \newline
(ii) a monomorphism,  if $L/K$ is unramified and $p\nmid n$.
\end{Th}

\begin{Rm} \label{rm:22}
In the case (ii) we cannot say whether $i_{*}$ is surjective or not, since, in general, $trf$ is
only an epimorphism and not an isomorphism.
\end{Rm}

\begin{Cr} \label{cr:3}
The statement \eqref{eq:42} holds in the cases (i)-(ii) of Theorem \ref{th:24}.
\end{Cr}
\begin{proof}
The statement is obvious in the case (i), since $i_{*}$ is a Galois isomorphism. For (ii) we use
the isomorphism of Theorem \ref{th:23} for $S$ and $R$   noting  that {
\begin{equation*}
R/(1-\varphi)R\cong \mathbb{Z}_{p}\cong S/(1-\varphi)
\end{equation*}}

in this situattion, hence $SK_{1}(S[G])$ is isomorphic to $SK_{1}(R[G])$ (as an abstract finite
group). Since $i_{*}$ is a Galois monomorphism in the case (ii), we get the statement.
\end{proof}

\begin{Rm} \label{rm:23}
Corollary \ref{cr:3} generalizes immediately to the case of infinite algebraic extensions $L$.
\end{Rm}

To study more general rings (for example completions of infinite extensions of $\mathbb{Q}_{p}$) we
need generalizations of Oliver's results on $SK_1$ to such rings as are announced to appear in
\cite{CPT. 1} in a very general setting. Meanwhile  we outline an ad hoc description sufficient for
our purposes. {  We just note that the same arguments as used below also should work for any ring
$R$ as in the beginning of subsection \ref{axiomatic} and satisfying the surjectivity of $1-F.$}

Let $p$ be an odd prime number. For the rest of this section we assume that $R$ is the ring of Witt vectors
 of a $p$-closed algebraic extension $\kappa$ of $\mathbb{F}_p,$ i.e.\   $\kappa$ does not allow any extension of degree $p.$
 The main example we have in mind being $\widehat{\mathbb{Z}_p^{ur}}=W(\overline{\mathbb{F}_p}).$ We note, that such a ring satisfies Hypothesis in \cite{CPT} and is a discrete valuation ring. We write $\mathfrak{m}$ for its maximal ideal
 $pR$  and  start with a crucial (certainly well-known)
observation:

\begin{Le}\label{sur}
We have an exact sequence \[\xymatrix@C=0.5cm{
  0 \ar[r] & {\mathbb{Z}_p}  \ar[rr]^{ } && R \ar[rr]^{1-\varphi} && R \ar[r] & 0, }\]
where $\varphi$ denotes the Frobenius endomorphism of $R$.
\end{Le}

\begin{proof}
By Artin-Schreier theory and the $p$-closeness of $\kappa$ we have the obvious exact sequence
\[\xymatrix@C=0.5cm{
  0 \ar[r] & {\mathbb{F}_p}  \ar[rr]^{ } && {\kappa}\ar[rr]^{1-\varphi} && {\kappa} \ar[r] & 0, }\]
  because $(1-\varphi)(x)=x-x^p.$ Inductively, one shows that, for all $n\geq1,$ also
\[\xymatrix@C=0.5cm{
  0 \ar[r] & {\mathbb{Z}_p/p^n\mathbb{Z}_p}  \ar[rr]^{ } && R/p^nR \ar[rr]^{1-\varphi} && R/p^nR \ar[r] & 0 }\]
is exact. Thus for any given $r=(r_n)_n\in \projlim_{n} R/p^nR=R$ the   sets $S_n:=\{s_n\in R/p^nR
| (1-\varphi)(s_n)=r_n\}$ are finite and form an inverse system, whence $S:=\projlim_n S_n$ is
non-empty and any $s\in S$ satisfies $(1-\varphi)(s )=r $ by construction.
\end{proof}

Let $G$ be a finite $p$-group. The split exact sequence
\[ \xymatrix@C=0.5cm{
  1 \ar[r] & I(R[G]) \ar[rr]  && R[G] \ar[rr]^{ } && R \ar[r] & 1 }\]
  induces isomorphisms
\[K_1(R[G])\cong K_1(R[G],I(R[G]))\oplus R^\times\] and
\[R[\mathcal{C}_G]\cong \phi(I(R[G]))\oplus R,\]
where $\phi: R[G]\twoheadrightarrow R[\mathcal{C}_G]$ denotes the canonical map, $\mathcal{C}_G$
denoting the conjugacy classes of $G.$ By $\mathrm{Log}(1-x)$ we denote the logarithm series. Then
the map $\frac{1}{p}\mathcal{L}=\phi(\frac{1}{p}(p-\Psi)(\mathrm{Log}(1-x)))$ defined on page 13 of
\cite{CPT} induces by \cite[cor.\ 3.3, thm.\ 3.17]{CPT} and the lemma above a surjective map
\[\Gamma_{I(R[G])}:K_1(R[G],I(R[G]))\twoheadrightarrow \phi(I(R[G])). \]
We use a generalization of Theorem 2.8 in \cite{O.} (for ideals contained in the Jacobson radical) in order to show, that this map is actually a group homomorphism, which together with the surjective homomorphism
\[\Gamma_R:R^\times \twoheadrightarrow R,\]
which sends $x\in 1+\mathfrak{m}$ to $\frac{1}p{}(p-\varphi)\mathrm{Log}(x)$ and
$x\in\kappa^\times$ to zero (note that $\mathrm{Log}(1+pR)=pR$ and that $p-\varphi$ is an
isomorphism of $R$), defines a surjective group homomorphism
\[\Gamma_{R[G]}=\Gamma_{I(R[G])}\oplus \Gamma_R :K_1(R[G])\twoheadrightarrow R[\mathcal{C}_G],\]
which factorizes over
\[\Gamma_{\mathrm{Wh}(R[G])}:\mathrm{Wh}(R[G]):= K_1(R[G])/(G^{ab}\times \mu_R)\twoheadrightarrow R[\mathcal{C}_G].\]
Setting
\[SK_1'(R[G]):=\ker(\Gamma_{\mathrm{Wh}(R[G])})\]
we obtain the following exact sequence
\begin{equation} \label{eq:41}
\xymatrix@C=0.5cm{
  1 \ar[r] & SK_1'(R[G]) \ar[rr]^{ } && {\mathrm{Wh}(R[G])}\ar[rr]^{\Gamma_{\mathrm{Wh}(R[G])}} && R[\mathcal{C}_G] \ar[r] & 1 .}
\end{equation}

  The relation between $SK_1'(R[G])$ and the original $SK_1(R[G])$ will
be cleared below.

Our goal is to prove the following

\begin{Th}
Let $G$ be a $p$-group. Then $SK_1'(R[G])=1.$ In particular, \[\mathrm{Wh}(R[G])\cong
R[\mathcal{C}_G]\] is torsion free and a Hausdorff topological group (the second group being a
pseudocompact $R$-module).
\end{Th}

\begin{proof}
The proof proceeds by induction on the order of $G.$ If $G$ is trivial, it is well-known that the
$SK_1'(R)=1,$ because the kernel of $\Gamma_R$ is just $\mu_R.$ Now assume $G$ to be non-trivial.
Then there exists a central element $z\in G$ of order $p.$ We set $\bar{G}:=G/<z>$ and write
$\alpha: G \twoheadrightarrow \bar{G}$ for the canonical projection. Consider the following
commutative diagram with exact rows
\[\xymatrix{
   &     & 0\ar[d]_{ } & 0\ar[d]_{ }   &   \\
   &   & {\ker(\mathrm{Wh}(\alpha))} \ar[d]_{ } \ar[r]^{ } & (1-z)R[\mathcal{C}_G] \ar[d]_{ } \ar[r]^{ } &  0  \\
  0   \ar[r]^{ } & SK_1'(R[G]) \ar[d]_{SK_1'(\alpha)} \ar[r]^{ } & {\mathrm{Wh}(R[G])} \ar[d]_{\mathrm{Wh}(\alpha)} \ar[r]^{\Gamma_\mathrm{Wh}} & R[\mathcal{C}_{ {G}}] \ar[d]_{H_0(\alpha)} \ar[r]^{ } & 0  \\
  0 \ar[r]^{ } & SK_1'(R[\bar{G}]) \ar[r]^{ } & {\mathrm{Wh}(R[\bar{G}])} \ar[r]^{\Gamma_\mathrm{Wh}} & R[\mathcal{C}_{\bar{G}}] \ar[r]^{ } & 0,}\]
in which also the right column is exact by \cite[lem.\ 3.9]{CPT}. Here $I_z$ denotes the ideal
$(1-z)R[G].$

An immediate generalization of \cite[prop.\ 6.4]{O.} to our setting tells us that the logarithm
induces an exact sequence
\[\xymatrix@C=0.5cm{
  1 \ar[r] & <z> \ar[rr]^{ } && K_1(R[G],I_z) \ar[rr]^{\mathrm{log}} && {\mathrm{H}_0(G,I_z)} \ar[r] & 0 }\]
(note that $\tau$ in (loc.\ cit.) has to be replaced by the trivial map, because $1-\varphi$ in
Lem.\ 6.3 is surjective on $\kappa;$ also \cite[thm.\ 2.8]{O.} (for ideals contained in the Jacobson radical)
 needed for the proof generalizes immediately to our setting, because by \cite[(20.4)]{L.} $S:=M_n(R[G])$ is also
semilocal and satisfies $J(S)^N\subseteq pS$ for $N$ sufficiently big).

Thus we obtain a commutative diagram with exact rows
\[\xymatrix{
  0   \ar[r]^{ } & B\ar[d]_{ } \ar[r]^{ } & K_1(R[G],I_z)/<z> \ar[d]_{\mathrm{log}}^{\cong} \ar[r]^{} & {\ker(\mathrm{Wh}(\alpha))} \ar[d]_{\Gamma_{R[G]}} \ar[r]^{ } & 0   \\
  0 \ar[r]^{ } & {\mathrm{H}_0(G,(1-z)R[\Omega])} \ar[r]^{ } & {\mathrm{H}_0(G,(1-z)R[G])} \ar[r]^{ } & (1-z)R[\mathcal{C}_{G}] \ar[r]^{ } & 0.   }\]

By the Snake-lemma we see that
\begin{eqnarray*}
\ker(SK_1'(\alpha))&=& \ker(\mathrm{Wh}(\alpha))\cap SK_1'(R[G])\\
&=&{\mathrm{H}_0(G,(1-z)R[\Omega])}/\mathrm{log}B\\
&=& R[\Omega]/\psi^{-1}(\mathrm{log}B),
\end{eqnarray*}
where $\psi:R[\Omega]\twoheadrightarrow {\mathrm{H}_0(G,(1-z)R[\Omega])}$ is induced by
multiplication with $(1-z).$  Here, following \cite{O.}  we write
\[\Omega =\{g\in G|  g \mbox{ is conjugate to } zg\}\] and note that our last term in the above equation
corresponds to $D/C$ in the proof of \cite[thm.\ 7.1]{O.}. Hence, copying literally the same
arguments and noting again that $1-\varphi$ is surjective on $R,$ we see from (c) on p. 176 in
(loc.\ cit.) that $C=\psi^{-1}(\mathrm{log}B)=R[\Omega].$ In other words, $\ker(SK_1'(\alpha))$ is
trivial. Since, by our induction hypothesis also $SK_1'(R[\bar{G}])$ vanishes, the theorem is
proved.
\end{proof}

Let $L^0$ be the unique unramified algebraic extension of $\mathbb{Q}_p$ with residue field
$\kappa.$

\begin{Cr} \label{cr:31}
For any open subgroup $\Delta\subseteq Gal(L^0/\mathbb{Q}_p)$, there is an exact sequence
\[\xymatrix@C=0.5cm{
  1 \ar[r] & SK_1(R^\Delta[G]) \ar[rr]^{ } && K_1(R^\Delta [G])\ar[rr]^{ } && K_1(R[G])^\Delta \ar[r] & 1 ,}\]
and   isomorphisms
\[  {\mathrm{H}^1(\Delta, \mu_R)} \cong {\mathrm{H}^1(\Delta, K_1(R  [G]))}  \]
of continuous cochain cohomology groups and {  coinvariants \[(\mu_R)_\Delta\cong
K_1(R[G])_\Delta,\] respectively.}
\end{Cr}

\begin{proof}
Taking $\Delta$-invariants of the exact sequence (of topological Hausdorff modules, cp.\ Corollary
\ref{cr:32} for $K_1$)
\[\xymatrix@C=0.5cm{
  1 \ar[r] & G^{ab}\times \mu_R \ar[rr]^{ } && K_1(R[G]) \ar[rr]^{\Gamma_{R[G]}} && R[\mathcal{C}_G] \ar[r] & 0 }\]

and noting the Galois invariance of $\Gamma_{R[G]}$ (if we choose the arithmetic Frobenius in its
definition) we obtain the following commutative diagram with exact rows (the first of which is the
standard exact sequence as proved in \cite{O.})

\begin{tiny}
\[\xymatrix{
  1  \ar[r]^{ } & G^{ab}\times \mu_{R^\Delta}\times SK_1(R^\Delta[G]) \ar[d]_{ } \ar[r]^{ } & K_1(R^\Delta [G]) \ar[d]_{ } \ar[r]^{ } & R^\Delta[\mathcal{C}_G] \ar[d]^{\cong } \ar[r]^{ } & G^{ab} \ar@^{^(->}[d]_{ } \ar[r]^{ }& 1\\
  1 \ar[r]^{ } & G^{ab}\times {\mu_R}^\Delta \ar[r]^{ } & K_1(R[G])^\Delta \ar[r]^{} & R[\mathcal{C}_G]^\Delta \ar[r]^{ } & {\mathrm{H}^1(\Delta, \mu_R) \times G^{ab}} \ar[r]^{ }& \ldots  }\]
\end{tiny}
from which the claim follows using \cite[prop.\ 1.6.13]{NSW}. Note that
\begin{eqnarray*}
{\mathrm{H}^1(\Delta, R [\mathcal{C}_G])} &\cong & \projlim_{n} \mathrm{H}^1(\Delta,
R/p^nR)^{\#\mathcal{C}_G} \\
&\cong & \projlim_{n}  ( R/p^nR)_\Delta^{\#\mathcal{C}_G}=0
\end{eqnarray*}
by  the straight forward generalization of \cite[thm.\ 2.3.4]{NSW} to pseudocompact modules, again
\cite[prop.\ 1.6.13]{NSW} and  Lemma \ref{sur}. { Alternatively, we may replace the long exact
cohomology sequence above by the kernel/cokernel exact sequence arising from the Snake-lemma
associated to multiplication by $1-\tau$ for any topological generator $\tau$ of $\Delta.$}
\end{proof}

\begin{Cr} \label{cr:32}
$SK_1(R[G])=1,$ i.e., $K_1(R[G])\cong \mathrm{Det}(R[G]^\times).$ In particular, \[K_1(R[G])\cong
\projlim_{n} K_1(R/p^nR[G]).\]
\end{Cr}

The last claim follows from the first one using  \cite[prop.\ 1.3]{CPT}. For the proof of the first
claim of the corollary consider the following diagram

\begin{equation}
\label{SK_1} \xymatrix{
    &  & 1 \ar[d]_{ }  &  &   \\
  &  & SK_1(R[G]) \ar[d]_{ }  &  &   \\
 1   \ar[r]^{ } & G^{ab}\times \mu_R   \ar[r]^{ } & K_1(R[G]) \ar[d]_{\mathrm{Det}} \ar[r]^{\Gamma_{R[G]}} & R[\mathcal{C}_G] \ar@^{^(->}[d]_{\mathrm{Tr}}
  \ar[r]^{ } & 1 ,  \\
   &   & {\mathrm{Hom}(R_G,\mathbb{C}_p^\times)} \ar[r]^{\Gamma_\mathrm{Hom}} & {\mathrm{Hom}(R_G,\mathbb{C}_p )}   &     }
\end{equation}

where $\Gamma_\mathrm{Hom}$ is defined as follows: Choose any continuous lift $F:\mathbb{C}_p \to
\mathbb{C}_p$ of the absolute Frobenius automorphism and denote by $\log:\mathbb{C}^\times \to
\mathbb{C}_p$ the usual $p$-adic logarithm. We write $\psi^p$ and $\psi_p$ for the $p$th Adams
operator, which is characterized by \[tr(g,\psi^p\rho)=tr(g^p,\rho) \mbox{ for all } g\in G,\] and
its adjoint, respectively. Also the rule $f(\rho) \mapsto F(f(\rho^{F^{-1}}))$ induces an operator
$\tilde{F}$ on $\mathrm{Hom}(R_G,\mathbb{C}_p^\times ) $ and $\mathrm{Hom}(R_G,\mathbb{C}_p ),$
which commutes obviously with $\psi_p.$ Now we set
 \[\Gamma_\mathrm{Hom}(f) =\frac{1}{p}(p-\tilde{F}\psi_p)(\log\circ f ).\]
 Finally, $\mathrm{Tr},$   the additive analog of $\mathrm{Det},$  is induced by \[\mathrm{Tr}(\lambda)(\rho)= \mathrm{tr} (\rho(\lambda)),
 \] where $\rho:R[G]\to M_n(\mathbb{C}_p)$ is the linear extension of $\rho:G\to
  GL_n(\mathbb{C}_p)$ keeping the same notation.  One easily checks that

 \begin{equation}\label{DET}
 \mathrm{Det}( F(\lambda))(\rho)=\tilde{F}(\mathrm{Det}(\lambda))(\rho)=F
 \mathrm{Det}(\lambda)(\rho^{F^{-1}})
 \end{equation}
and
\begin{equation}\label{TR}
 \mathrm{Tr}( \Psi(\lambda))(\rho)=\tilde{F}(\mathrm{Tr}(\psi_p\lambda))(\rho)=F
 \mathrm{Tr}(\lambda)(\psi_p\rho^{F^{-1}})
 \end{equation}

The Corollary will follow immediately from the following

\begin{Le}
The diagram \eqref{SK_1} commutes and $\mathrm{Det}$ restricted to $G^{ab}\times \mu_R$ is
injective.
\end{Le}

\begin{proof} The injectivity being well-known we only
check the commutativity similarly as in \cite[prop.\ 4.3.25]{Sn}. Since $K_1(R[G], I(R[G]))$ and $K_1(R)$ generate $K_1(R[G])$ as
has been observed above, it suffices to check this individually on each direct summand. The case of
$K_1(R)$ being similar but easier, we assume that $a$ belongs to $K_1(R[G], I(R[G]))$ and calculate
using the definitions, \eqref{TR}, the continuity of $\rho,$ the fact that $\log$ transforms $\det$
into $\mathrm{tr}$ and \eqref{DET}:
\begin{eqnarray*}
(\mathrm{Tr}\circ \Gamma (a))(\rho)&=&\mathrm{Tr}(\frac{1}{p} (p-\Psi) \log(a))(\rho) \\
&=& \mathrm{Tr}(\log(a))(\rho) -\frac{1}{p}  F \mathrm{Tr}(\log(a)(\psi_p \rho^{F^{-1}}) \\
&=& \mathrm{tr}(\log \rho(a))-\frac{1}{p}  F \mathrm{tr} \log(\psi_p \rho^{F^{-1}}(a))\\
&=& \log (\det(\rho(a)))-\frac{1}{p} F \log (\det(\psi_p \rho^{F^{-1}}(a))\\
&=&\frac{1}{p} \{ p \log \mathrm{Det}(a)(\rho) - F\log \mathrm{Det}(a)(\psi_p \rho^{F^{-1}})\}\\
&=&\frac{1}{p}(p-\tilde{F}\psi_p) \log \mathrm{Det}(a)(\rho)\\
&=& (\Gamma_\mathrm{Hom}\circ \mathrm{Det} (a))(\rho).
\end{eqnarray*}
\end{proof}

\begin{Rm}
The case $p=2$ should be treated separately, since
$\mathrm{Log}(R^{\times})=\mathrm{Log}(1+2R)=4R+2(1-\varphi)R\cong 2R$, so that we have to replace $R[\mathcal{C}_G]$ by $2R[\mathcal{C}_G]$ in the exact sequence \eqref{eq:41}, or, if we keep $R[\mathcal{C}_G]$ in \eqref{eq:41}, then we get a finite cokernel, which we denote by $\mu=\left\langle -1\right\rangle$. Doing required corrections we can proof
Corollaries \ref{cr:31} and \ref{cr:32} also in this case. Note, that the finite cokernel $\mu$ also appears in the first row of the commutative diagram in the proof of Corollary \ref{cr:31} (see \cite[thm.\ 6.6]{O.}).
\end{Rm}

\begin{Le}
The statement (i) of Theorem \ref{th:15} is true also for $R$ being the ring of Witt vectors of a $p$-closed algebraic extension $\kappa$ of $\mathbb{F}_{p}$.
\end{Le}
\begin{proof}
The injectivity is obvious $SK_{1}(R[G])$ being trivial (see Corollary \ref{cr:32}) and the surjectivity follows from the generalized Proposition 15 in \cite{O. 2}.
\end{proof}

Corollary \ref{cr:32} and the lemma above imply
\begin{Cr} \label{cr:43}
$SK_{1}(S[G])=1$ for any totally ramified  integral extension $S$ of $R$, where $R$ is the ring of
Witt vectors of a $p$-closed algebraic extension $\kappa$ of $\mathbb{F}_{p}$.
\end{Cr}

Finally, we want to generalize Corollaries \ref{cr:2} and \ref{cr:43} to the case of an arbitrary finite group $G$. For this we need

\begin{Th} \label{th:32}
Fix a prime $p$ and a Dedekind domain $R$ with field of fractions $K$, such that
$\mathbb{Q}_{p}\subseteq K \subseteq \mathbb{C}_{p}$. For any finite group $G$, let
$g_{1},\ldots,g_{k}$ be $K$-conjugacy class representatives for elements in $G$ of order
$n_{i}=\mathrm{ord}(g_{i} )$ prime to $p$, and set
\begin{equation*}
N_{i}=N^{K}_{G}(g_{i})=\{x\in G:\;xg_{i}x^{-1}=g^{a}_{i} \text{, some } a\in
Gal(K(\zeta_{n_{i}})/K)\}\;\;
\end{equation*}
 and $Z_{i}=\mathcal{C}_{G}(g_{i})$. Then $ SK_{1}(R[G])$ is computable for induction from $p$-elementary subgroups of $G$ and there is an isomorphism
\begin{equation*}
 SK_{1}(R[G])\cong \bigoplus^{k}_{i=1} \mathrm{H}_{0}\big(N_{i}/Z_{i};\underset{H\in \mathcal{P}(Z_{i})}{\varinjlim} SK_{1}(R[\zeta_{n_{i}}][H])\big),
\end{equation*}
where $\mathcal{P}(Z_{i})$ is the set of $p$-subgroups of $Z_i$ and $R[\zeta_{n_{i}}]$ denotes the
integral closure of $R$ in $K(\zeta_{n_{i}})$.
\end{Th}

\begin{proof}
See \cite[thm.\ 11.8 and thm.\ 12.5 ]{O.}
\end{proof}

\begin{Cr} \label{cr:33}   Let $G$ be an arbitrary finite group. Let $S$ be $\mathcal{O}_{M}$, with $M$ as in Corollary \ref{cr:2}, but of finite absolute ramification index over $\mathbb{Q}_{p}$ (resp. $S$ be a finite totally ramified extension of $R$, where $R$ is the ring of Witt vectors of a $p$-closed algebraic extension $\kappa$ of $\mathbb{F}_{p}$), then $SK_{1}(S[G])=1$.
\end{Cr}

\begin{proof}
Use Corollaries \ref{cr:2}, \ref{cr:43} and Theorem \ref{th:32}. Note also, that $R[\zeta_{n_{i}}]$
is a finite unramified extensions of $R$.
\end{proof}

\begin{flushleft}
\textbf{End of the $ SK_{1}$-part.}
\end{flushleft}

  For a finite group $G$ and $S$ being the ring of integers of either an
arbitrary tamely ramified extension $L$ of $\mathbb{Q}_p$ (type 1) or the completion of a tamely
ramified extension $L$ of finite absolute ramification index over $\mathbb{Q}_{p}$, whose residue
field is p-closed (type 2). Let $R=S^\Delta$ be the fixed ring of $S$, where $\Delta$ is an open
subgroup of $Gal(L/\mathbb{Q}_p)$ containing the inertia group. We write $K$ and $C$ for the kernel
and cokernel of the map \[i_*:SK_1(R[G])\to SK_1(S[G])^\Delta,\] respectively. Note that they are
finite $p$-primary abelian groups. By $K_1(R[G])_\mathbb{Q}$ we denote rational $K$-groups
$K_1(R[G])\otimes_\mathbb{Z}\mathbb{Q}$. The Snake-lemma, Theorem  \ref{th:34} and the
$SK_{1}$-part
 imply immediately

\begin{Th}\label{prop:4} Let $p^n$ be the $p$-part of $[L:L^{\Delta}]$ $(0\leq n\leq \infty)$.
\begin{enumerate}
\item The following sequence is exact \[\xymatrix@C=0.5cm{
  1 \ar[r] & K \ar[rr]^{ } && K_1(R[G]) \ar[rr]^{i_* } && K_1(S[G])^{\Delta} \ar[rr]^{ } && C \ar[r] & 1
  }\] and induces \[K_1(R[G])_\mathbb{Q} \cong K_1(S[G])_\mathbb{Q}^\Delta.\]

\item If $S$ is of type 1 and $n=0$, then
\[K=1\;\;\;\text{ and }\;\;\;C=1.\]

\item If $S$ is either of type 2 or of type 1 with $n=\infty$, then we have
\[K\cong SK_1(R[G])\;\;\;\text{ and }\;\;\;C=1 .\]

\item Let $G$ be a $p$-group and $S$ be of type 1 with $0 < n < \infty$, then
 \[K\cong SK_1(R[G])[p^n]\;\;\;\text{ and }\;\;\;C\cong SK_1(R[G])/p^n.\]
\end{enumerate}
\end{Th}

\section{The case of residue class fields.}

\quad Let $\lambda$ be an arbitrary (not necessary finite) Galois extension of $\mathbb{F}_{p}$ and
$G$ be a finite group. Let $\phi$ denote the Frobenius automorphism on $\lambda$, which takes $x\in
\lambda$ to $x^{p}$, then $Gal(\lambda/\mathbb{F}_{p})=\overline{\left\langle \phi\right\rangle}$.
Moreover, if $\mathbb{F}_{p^{n}}\subset\lambda$, then
$\mathbb{F}_{p^{n}}=\lambda^{\overline{\left\langle \phi^{n}\right\rangle}}$. We fix such an $n,$
set $\kappa=\mathbb{F}_{p^{n}}$ and $\Delta=\overline{\left\langle \phi^{n}\right\rangle}$.  We are
going to prove the following

\begin{Th} \label{th:10}
With the notation as above, we have an exact sequence

\[\xymatrix@C=0.5cm{
  1 \ar[r] & K \ar[rr]^{ } && {K_{1}(\kappa[G])} \ar[rr]^{i_* } && {K_{1}(\lambda[G])^{\Delta}} \ar[rr]^{ } && C \ar[r] & 1,
  }\]
which induces
\begin{equation*}
K_{1}(\kappa[G])_\mathbb{Q}\cong K_{1}(\lambda[G])^{\Delta}_\mathbb{Q},
\end{equation*}
where $K$ and $C$ are as in Proposition \ref{prop:4} for $R$ and $S$ the unique unramified
extensions of $\mathbb{Z}_p$ lifting $\kappa$ and $\lambda,$ respectively. As usual  $\phi$ (resp.
$\phi^{n}$) acts on the $K_{1}$-groups coefficientwise.
\end{Th}

\begin{proof}

From \cite{SV} we have an exact sequence
\begin{equation} \label{eq:13}
0\rightarrow \mathbb{Z}_{p}[\mathcal{C}_G]\rightarrow K_{1}(\mathbb{Z}_{p}[G])\rightarrow
K_{1}(\mathbb{F}_{p}[G])\rightarrow 1,
\end{equation}
where $\mathbb{Z}_{p}[\mathcal{C}_G]$ is a finitely generated free $\mathbb{Z}_{p}$-module over the
set of conjugacy classes in $G$.

With the same argument as in \cite{SV}, we can obtain \eqref{eq:13} for finite unramified extensions of $\mathbb{Q}_{p}$ and their rings of integers, and, since $K_{1}$ commutes with the direct limit, also for infinite unramified extensions. Using this fact we get the following commutative diagram with exact rows
\begin{equation*}
\begin{array}{ccccccccl}
& & & & 1 & & & & \\
& & & & \Big\downarrow & & & & \\
& & & & K & & & & \\
& & & & \Big\downarrow & & & & \\
0 & \rightarrow & R[\mathcal{C}_G] & \rightarrow & K_{1}(R[G]) & \rightarrow & K_{1}(\kappa[G]) & \rightarrow & 1\\
& & \Big\downarrow & & \Big\downarrow & & \Big\downarrow & & \\
0 & \rightarrow  & S[\mathcal{C}_G]^{\Delta} & \rightarrow & K_{1}(S[G])^{\Delta} & \rightarrow & K_{1}(\lambda[G])^{\Delta} & \rightarrow & \mathrm{H}^{1}(\Delta,S[\mathcal{C}_G]),\\
& & & & \Big\downarrow & & & & \\
& & & & C & & & & \\
& & & & \Big\downarrow & & & & \\
& & & & 1 & & & & \\
\end{array}
\end{equation*}
where the bottom row is the part of the long exact sequence in the group cohomology
associated to the short exact sequence of $\Delta$-modules, and $\Delta$ acts coefficientwise.

The left hand side vertical map is an isomorphism, as $R[\mathcal{C}_G]$ and $S[\mathcal{C}_G]$ are
finitely generated free $R$- and $S$-modules respectively. The middle column is exact by Prop.
\ref{prop:4}. Thus, to prove the theorem, it is enough to show, that
$\mathrm{H}^{1}(\Delta,S[\mathcal{C}_G])=1$.

From \cite{FS} we know, that $H^{1}(Gal(M_{1}/M_{2}),\mathcal{O}_{M_{1}})=1$ for every finite unramified extension
$M_{1}$ of $\mathbb{Q}_{p}$. Since $\Delta$ can be written as the inverse limit of finite groups
corresponding to the finite unramified subextensions of $S$, and $S[\mathcal{C}_G]$ as a
$\Delta$-module is isomorphic to the direct sum of copies of $S$, we get the statement above by
using standard properties of the cohomology of groups. \end{proof}

\end{document}